\documentclass[a4paper, 10pt]{amsart}
\usepackage{amsthm}
\usepackage[]{amsmath}
\usepackage{amssymb}
\usepackage{enumerate}
\usepackage{tabularx}
\usepackage[]{color}
\usepackage[left=3cm, right=3cm]{geometry}
\usepackage[colorlinks]{hyperref}
\usepackage{tikz}
\usepackage{multirow}
\usepackage{subcaption}
\usepackage{algorithm}
\usepackage{algorithmic}

\newcommand{\bm}[1]{\boldsymbol{#1}}

\newcommand{\lj}{[ \hspace{-2pt} [}
\newcommand{\rj}{] \hspace{-2pt} ]}
\newcommand{\mb}[1]{\mathbb{#1}}
\newcommand{\mc}[1]{\mathcal{#1}}
\newcommand{\mr}[1]{\mathrm{#1}}
\newcommand{\jump}[1]{\lj #1 \rj}
\newcommand{\aver}[1]{ \{#1\}  }
\newcommand{\wt}[1]{ \widetilde{ #1}}

\newcommand\MTh{\mc{T}_h}
\newcommand\MEh{\mc{E}_h}
\newcommand\un{\bm{\mr n}}
\renewcommand{\d}[1]{\mathrm d \boldsymbol{#1}}
\newcommand{\unorm}[1]{ \| #1 \|_{\bm{\mr{V}}}}
\newcommand{\snorm}[1]{ \| #1 \|_{\bm{\mr{\Sigma}}}}
\newcommand\ddiv{\ifmmode \mathrm{div} \else \text{div}\fi}
\newcommand\tr{\ifmmode \mathrm{tr} \else \text{tr} \fi}

\newtheorem{assumption}{Assumption}
\newtheorem{theorem}{Theorem}
\newtheorem{lemma}{Lemma}

\title[Linear Elasticity]{A Least Squares Method for Linear Elasticity
using A Patch Reconstructed Space}

\author[R. Li]{Ruo Li} \address{CAPT, LMAM and School of Mathematical
  Sciences, Peking University, Beijing 100871, P.R. China}
\email{rli@math.pku.edu.cn}

\author[F.-Y. Yang]{Fanyi Yang} \address{School of Mathematical
  Sciences, Peking University, Beijing 100871, P.R. China}
\email{yangfanyi@pku.edu.cn}

\begin{document}
\maketitle

% vim:spell:tw=70:fo+=Mn:cc=70
\begin{abstract}
  We propose a discontinuous least squares finite element method for
  solving the linear elasticity. The approximation space is obtained
  by patch reconstruction with only one unknown per element. We apply
  the $L^2$ norm least squares principle to the stress-displacement
  formulation based on discontinuous approximation with normal
  continuity across the interior faces. The optimal convergence order
  under the energy norm is attained. Numerical results of linear
  elasticity are presented to verify the error estimates. In addition
  to enjoying the advantages of discontinuous Galerkin method, we
  illustrate the great simplicity in implementation, the robustness
  and the improved efficiency of our method.

  \noindent \textbf{keywords}: Linear elasticity, Least squares
  method, Patch reconstruction, Discontinuous Galerkin method.

\end{abstract}

%%% Local Variables:
%%% mode: latex
%%% TeX-master: "linear_elasticity"
%%% End:

% vim:spell:tw=70:fo+=Mn:cc=70
\section{Introduction}
The stress-displacement formulation is the first-order system of the
linear elastic problem providing a relation between the strain and the
equilibrium equation. Compared to the pure displacement formulation,
the stress-displacement formulation is more attractive when
considering the nearly incompressible case. For many practical
applications, the stress field is the quantity of particular
interests. Therefore, there are many efforts devoted to the mixed
finite element method based on the weak form of the
stress-displacement formulation. We refer to \cite{Adams2005mixed,
Arnold2002mixed, Arnold2008symmetric, Hu2015new, Hu2015family} for
conforming mixed elements and \cite{Hu2016simplest,
Arnold2014nonconforming, Yi2006new} for non-conforming elements and
the references therein. The main challenge of the mixed finite element
method is the construction of proper finite element spaces since the
requirement of the stable combination of approximation spaces and the
symmetric constraint of the stress tensor. It has been a long-standing
open problem until Hu and Zhang's recently significant progress
\cite{Hu2015new, Hu2015family}. Their work requires a subtle
structure on the geometry of the element to construct the finite
element space particularly for the linear elastic problem.

To solve the linear elastic problem using common finite element
spaces, the least squares finite element methods have been
investigated in a sequence of papers \cite{Cai2003first, Cai2004least,
  Cai1998first, Bramble2001least, Starke2011analysis}. The least
squares finite element method is a sophisticated technique in
numerical partial differential equations, and we refer to the survey
paper \cite{Bochev1998review} and the references therein. The least
squares method based on a discrete minus one inner product is proposed
in \cite{Bramble2001least}. Cai and his coworkers developed the least
squares finite element methods based on the $L^2$ norm residual for
solving the stress-displacement system \cite{Cai2003first,
  Cai2004least, Cai1998first}. One of their advantages over the usual
mixed finite element method is the selection of the approximation
spaces is not subject to the stability condition.

In this paper, a new discontinuous least squares finite element method
is proposed based on the stress-displacement formulation. The novel
point is the new approximation space which is obtained by patch
reconstruction with one unknown per element \cite{li2016discontinuous,
li2012efficient, li2017discontinuous}. The new space could be regarded
as a subspace of the common space used in discontinuous Galerkin
finite element method. We follow the idea in \cite{Cai2004least} to
define an $L^2$ norm least squares functional based on discontinuous
approximation spaces and we derive the optimal convergence order under
the energy norm. By a series of numerical examples, the error
estimates are verified. As in \cite{Cai2004least}, the $L_2$ error can
only be proved sub-optimal though numerical results show an optimal
convergence for odd orders in approximation to the displacement. The
implementation of our method is very convenient that the coding for
different polynomial degrees and meshes with elements in different
geometry are most reusable. We present an example on polygonal mesh to
show such advantages of our method. The main steps of the method are
detailed in Appendix, which may be helpful to implement for any high
order accuracy and any polygonal mesh in an easy manner. The numerical
results with very large Lam\'e parameter are presented to exhibit the
robustness in the incompressible limit. A remarkable advantage of this
new space is its great efficiency. We make an efficiency comparison
between our method and the method in \cite{Cai2004least} using
continuous finite element space. It is clear that our method uses
fewer degrees of freedom than continuous approximation to achieve the
same accuracy, and for higher order approximation, the advantage in
efficiency of our method is more significant.

The rest of this paper is organized as follows. In Section
\ref{sec:space}, we introduce the reconstruction operator and its
corresponding approximation space, and we also present the basic
properties of the approximation space. In Section \ref{sec:dlsfem}, we
propose the discontinuous least squares method for the
stress-displacement formulation and we derive the error estimate in
energy norm. In Section \ref{sec:numericalresults}, we present a
series of the numerical examples to verify the convergence of our
method, and we also make a comparison to illustrate the efficiency of
the proposed method.

%%% Local Variables:
%%% mode: latex
%%% TeX-master: "linear_elasticity"
%%% End:

% vim:spell:tw=70:fo+=Mn:cc=70
\section{Approximation Space}
\label{sec:space}
Let $\Omega \subset \mb R^d (d = 2, 3)$ be a bounded convex domain
with smooth boundary $\partial \Omega$. We denote by $\MTh$ a
collection of polygonal (polyhedral) elements which partition the
domain $\Omega$. We define all interior faces of $\MTh$ as $\MEh^i$
and denote by $\MEh^b$ the set of faces lying on $\partial \Omega$. We
then let $\MEh = \MEh^i \cup \MEh^b$ the set of all faces. Let $h_K$
be the diameter of the element $K$ and $h_e$ be the size of the face
$e$ and we denote $h = h_{\max} = \max_{K \in \MTh}h_K$ as the mesh
size. We assume that $\MTh$ is a shape-regular partition of $\Omega$
in the sense of that: there exist 
\begin{itemize}
  \item two positive numbers $N$ and $\sigma$ that are independent of
    the mesh size $h$;
  \item a compatible subdivision $\wt{\mc T}_h$ consisting of
    shape-regular simplexes;
\end{itemize}
such that
\begin{itemize}
  \item any polygonal (polyhedral) element $K \in \MTh$ admits a
    decomposition $\wt{\mc T}_{h|K}$ into less than $N$ shape-regular
    simplexes;
  \item the shape-regularity of $\wt{K} \in \wt{\mc T}_h$ reads
    \cite{ciarlet2002finite}: the ratio $h_{\wt{K}} / \rho_{\wt{K}}$
    is bounded by $\sigma$ where $\rho_{\wt{K}}$ denotes the radius of
    the largest ball inscribed in $\wt{K}$.
\end{itemize}
The above regularity assumptions, which are common in finite
difference scheme \cite{antonietti2013hp} and in DG framework
\cite{li2017discontinuous}, could bring many useful consequences: 
\begin{itemize}
  \item[M1] there exists a positive constant $\sigma_v$ that is
    independent of $h$  such that $\sigma_v \rho_K \leq \rho_e$ for
    any element $K \in \MTh$ and any face $e \subset \partial K$;
  \item[M2][{\it trace inequality}] there exists a constant $C$ that is
    independent $h$ such that 
    \begin{equation}
      \| v \|_{L^2(\partial K)}^2 \leq C \left( h_K^{-1}
      \|v\|_{L^2(K)}^2 + h_K \| \nabla v \|_{L^2(K)}^2 \right), \quad
      \forall v \in H^1(K);
      \label{eq:traceinequality}
    \end{equation}
  \item[M3][{\it inverse inequality}] there exists a constant $C$ that
    is independent $h$ such that 
    \begin{equation}
      \| \nabla v \|_{L^2(K)} \leq Ch_K^{-1} \| v \|_{L^2(K)}, \quad v
      \in \mb P_m(K),
      \label{eq:inverseinequality}
    \end{equation}
    where $\mb P_m(\cdot)$ is the polynomial space of degree $\leq m$.
\end{itemize}
We then define a reconstruction operator with the given partition
$\MTh$ as follows. First, in every element $K$ we assign a point
$\bm{x}_K$ as its corresponding collocation point. The choice of
$\bm{x}_K$ could be very flexible, particularly in this paper
$\bm{x}_K$ is specified as the barycenter of the element $K$. Second,
for each $K \in \MTh$ we would construct an element patch $S(K)$ which
contains $K$ itself and some elements around $K$. To be specific, for
element $K$, a threshold value $\# S(K)$ is given to control the
cardinality of $S(K)$, and we construct $S(K)$ recursively. Let
$S_0(K) = \left\{ K \right\}$ and we define $S_t(K)$ as follow:
\begin{displaymath}
  S_t(K) = \bigcup_{\tiny
  \begin{aligned}
    \widetilde{K} \in \MTh ,\ & \widehat{K} \in S_{t - 1}(K) \\
    \widetilde{K} \cap \widehat{K} &= e \in \mathcal E_h
  \end{aligned}
  } \widetilde{K}, \quad t = 1, 2, \cdots 
\end{displaymath}
We end the recursion if the cardinality of $S_t(K)$ is greater than
$\# S(K)$. Then, we calculate all distances between the collocation
points of every element in $S_t(K)$ and point $\bm{x}_K$. We select
the $\# S(K)$ smallest values and collect the corresponding elements
to form the patch $S(K)$. Obviously the cardinality of $S(K)$ is just
$\# S(K)$. In Appendix \ref{sec:acsp}, we show the details of the
algorithm of the construction of the element patch.

Further, we denote by $\mc I_K$ the set containing all collocation
points correspond to the elements in $S(K)$:
\begin{displaymath}
  \mc I_K \triangleq \left\{ \bm{x}_{\wt{K}}\ |\ \forall \wt{K} \in
  S(K)  \right\}.
\end{displaymath}
Then, for any function $g \in C^0(\Omega)$ and element $K \in \MTh$ we
seek a polynomial of degree $m$ defined on $S(K)$ by solving the
following least squares problem:
\begin{equation}
  \begin{aligned}
    \mc{R}_K g = \mathop{\arg \min}_{ p \in \mb P_m(S(K))} &\sum_{ \bm{x}
    \in \mc I_K} |p(\bm x) - g(\bm x)|^2 \\
  \text{s.t. } & p(\bm{x}_K) = g(\bm{x}_K). \\
  \end{aligned}
  \label{eq:lsproblem}
\end{equation}
The geometrical positions of the points in $\mc I_K$ totally decide
the existence and uniqueness of the solution to \eqref{eq:lsproblem}.
We follow \cite{li2012efficient} to make the following assumption:
\begin{assumption}
  For any element $K \in \MTh$ and $p \in \mb{P}_m(S(K))$, one has
  that
  \begin{displaymath}
    p|_{\mc{I}_K} = 0 \quad \text{implies} \quad p|_{S(K)} \equiv 0.
  \end{displaymath}
\end{assumption}
The assumption in fact excludes the case that all the points in
$\mc{I}_K$ lie on an algebraic curve and demands that the number $\#
S(K)$ should be greater than $\text{dim}(\mb P_m)$.

It must be notable that the solution to \eqref{eq:lsproblem} has a
linear dependence on the function $g$, which enables us to define a
linear reconstruction operator $\mc{R}$ for $g$:
\begin{displaymath}
  (\mc Rg)|_K = (\mc R_Kg)|_K, \quad \text{for } K \in \MTh.
\end{displaymath}
With $\mc{R}$, the function $g \in C^0(\Omega)$ is mapped into a
piecewise polynomial function of degree $m$ on $\MTh$. We denote by
$U_h$ the image of the operator $\mc{R}$. Further, we define
$w_K(\bm{x}) \in C^0(\Omega)$ as 
\begin{displaymath}
  w_K(\bm{x}) = \begin{cases}
    1, \quad & \bm{x} = \bm{x}_K, \\
    0, \quad & \bm{x} \text{ outside } K. \\
  \end{cases}
\end{displaymath}
It is clear that $U_h = \text{span}\left\{ \lambda_K\ |\ \lambda_K =
\mc{R} w_K \right\}$ and one could explicitly write the reconstruction
operator $\mc R$ as 
\begin{equation}
  \mc Rg = \sum_{K \in \MTh} g(\bm{x}) \lambda_K(\bm{x}), \quad
  \forall g \in C^0(\Omega).
  \label{eq:explicitly}
\end{equation}
In Appendix \ref{sec:sllsp}, we present an example of linear
reconstruction to illustrate the implementation of solving the least
squares problem \eqref{eq:lsproblem}.

Then, we would investigate the approximation property of the operator
$\mc{R}$. We define the constant $\Lambda(m, S(K))$ for all elements
as 
\begin{displaymath}
  \Lambda(m, S(K)) = \max_{p \in \mb P_m(S(K))} \frac{\max_{\bm{x} \in
  S(K)} |p(\bm{x})|}{\max_{\bm{x} \in \mc{I}_K} |p(\bm{x})|}.
\end{displaymath}
With $\Lambda(m, S(K))$, we could state the following estimates.
\begin{lemma}
  For any function $g \in C^0(\Omega)$ and element $K \in \MTh$, the
  following inequalities hold true:
  \begin{equation}
    \| \mc{R}_m g\|_{L^\infty(K)} \leq \Lambda(m, S(K)) \sqrt{ \#
    S(K)} \max_{\bm{x} \in \mc{I}_K} |g (\bm x)|,
    \label{eq:stabilityinfty}
  \end{equation}
  and
  \begin{equation}
    \|g - \mc{R}_K g\|_{L^\infty(K)} \leq  \Lambda_m     \inf_{p \in
    \mb{P}_m (S(K))} \| g - p\|_{L^\infty( S(K))},
    \label{eq:appinfty}
  \end{equation}
  where $\Lambda_m \triangleq \max_{K \in \MTh}\left\{1 + \Lambda(m,
  S(K)) \sqrt{ \# S(K)}\right\}$.
  \label{le:stabilityinf}
\end{lemma}
\begin{proof}
  The proof could be found in \cite[Theorem 3.3]{li2012efficient}.
\end{proof}
Under some mild and practical conditions on element patch $S(K)$,
$\Lambda_m$ could be bounded uniformly which plays a vital role in the
convergence estimate. We refer to \cite{li2012efficient,
li2016discontinuous} for these conditions and more detailed discussion
about the uniform upper bound.  One of the conditions we shall note is
that the number $\# S(K)$ shall be far greater than $\text{dim}(\mb
P_m)$.  In Section \ref{sec:numericalresults}, we list the values of
$\# S(K)$ with different $m$ for the numerical tests.

As a direct result of Lemma \ref{le:stabilityinf}, we could state the
following approximation properties of the operator $\mc{R}$.
\begin{theorem}
  Let $g \in H^{m+1}(\Omega)$, there exist constants $C$ that are
  independent of $h$ such that 
  \begin{equation} 
    \begin{aligned}
      \|g - \mc Rg\|_{H^q(K)} &\leq C \Lambda_m h_K^{m+1 - q} \|g
      \|_{H^{m+1}(S(K))}, \quad q = 0, 1, \\
      \|D^q(g - \mc{R}g)\|_{L^2(\partial K)} & \leq C \Lambda_m h_K^{m
      + 1 - q - 1/2} \|g\|_{H^{m+1}(S(K))}, \quad q = 0, 1.\\
    \end{aligned}
    \label{eq:localapp}
  \end{equation}
  \label{th:localapp}
\end{theorem}
\begin{proof}
  The proof directly follows from \cite[Lemma 4]{li2016discontinuous}
  and \cite[Assumption A]{li2016discontinuous}.
\end{proof}

%%% Local Variables:
%%% mode: latex
%%% TeX-master: "linear_elasticity"
%%% End:

% vim:spell:tw=70:fo+=Mn:cc=70
\section{Discontinuous Least Squares Finite Element Method}
\label{sec:dlsfem}
The problem concerned in this paper is the first-order system
formulation of the linear elasticity: {\it seek the stress $\bm \sigma
= (\sigma_{ij})_{d \times d}$ and the displacement $\bm u = (u_1,
\dots, u_d)^T$ such that} 
\begin{equation}
  \begin{aligned}
    \mc A\bm{\sigma} -\bm{\varepsilon}(u) & = \bm{0}, \quad\
    \text{in}\ \Omega, \\
    \nabla \cdot \bm{\sigma} + \bm{f} &= \bm 0, \quad\ \text{in}\
    \Omega, \\
    \bm{u} &= \bm{g}, \quad\  \text{on}\ \Gamma_D, \\ 
    \un\cdot \bm{\sigma} &= \bm{h}, \quad\  \text{on}\ \Gamma_N, \\
  \end{aligned}
  \label{eq:problem}
\end{equation}
where $\bm{f}$ is a given body force and $\bm{g}, \bm{h}$  are the
boundary conditions. $\Gamma_D$ and $\Gamma_N$ are two disjoint parts
of the boundary $\partial \Omega$ such that $\bar{\Gamma}_D \cup
\bar{\Gamma}_N = \partial \Omega$. For simplicity, $\Gamma_D$ is
assumed to be non-empty. We denote by $\bm{\varepsilon}(\bm{u}) =
(\varepsilon_{i,j}(\bm{u}))_{d \times d}$ the symmetric strain tensor:
\begin{displaymath}
  \varepsilon_{i,j}(\bm{u}) = \frac{1}{2} \left( \frac{\partial
  u_i}{\partial x_j} + \frac{\partial u_j}{\partial x_i} \right).
\end{displaymath}
The constitutive law with Lam\'e parameters $\lambda, \mu > 0$ is
expressed by  the linear operator $\mc{A}: \mb{R}^{d \times d}
\rightarrow \mb{R}^{d \times d}$:
\begin{displaymath}
  \mc A\bm{\tau} \triangleq \frac{1}{2 \mu} \left( \bm{\tau} -
  \frac{\lambda}{d \lambda + 2 \mu} (\tr\bm{ \tau}) \bm{I} \right),
  \quad \forall \bm{\tau} \in \mb{R}^{d \times d},
\end{displaymath}
where $\bm{I}$ is the identity operator and $\tr(\cdot)$ denotes the
standard trace operator.

Hereafter, let us note that $C$ and $C$ with a subscript that are
generic constants that may differ from line to line but are
independent of the mesh size $h$ and the parameter $\lambda$, and we
will use the standard notation and definition for the spaces $L^2(E),
L^2(E)^d$, $L^2(E)^{d \times d}$, $H^s(E)$, $H^s(E)^d, H^s(E)^{d
\times d}$ with $s \geq 0$ and $E$ a bounded domain, and their
associated inner products and norms. Let 
\begin{displaymath}
  H_D^1(\Omega) = \left\{ v \in H^1(\Omega)\ |\ v = 0, \ \text{on}\
  \Gamma_D \right\}, \quad   H_N^1(\Omega) = \left\{ v \in
  H^1(\Omega)\ |\ v = 0, \ \text{on}\ \Gamma_N \right\}.
\end{displaymath}
We denote by $H_D^{-1}(\Omega)$ the dual space of $H_D^1(\Omega)$ with 
the norm
\begin{equation}
  \| \phi \|_{-1, D} = \sup_{0 \neq \psi \in H_D^1(\Omega)}
  \frac{(\phi, \psi)}{ \| \psi\|_{H^1(\Omega)}}.
  \label{eq:nnormdef}
\end{equation}
Moreover we would use, for the partition $\MTh$, the standard broken
Sobolev spaces $H^s(\MTh)$, $H^s(\MTh)^d$, $H^s(\MTh)^{d \times d}$
with $s \geq 0$ and its corresponding broken norms. For the tensor
spaces $L^2(E)^{d \times d}$, $H^s(E)^{ d\times d}$, $H^s(\MTh)^{d
\times d}$, we define their corresponding symmetric spaces as
\begin{displaymath}
  \begin{aligned}
    L^2(E)^{\mb{S}, d \times d} &\triangleq \left\{ \bm{\tau} \in
    L^2(E)^{d \times d} \ |\ \bm{\tau} = \bm{\tau}^T \right\}, \\
    H^s(E)^{\mb{S}, d \times d} &\triangleq \left\{ \bm{\tau} \in
    H^s(E)^{d \times d} \ |\ \bm{\tau} = \bm{\tau}^T \right\}, \\
    H^s(\MTh)^{\mb{S}, d \times d} &\triangleq \left\{ \bm{\tau} \in
    H^s(\MTh)^{d \times d} \ |\ \bm{\tau} = \bm{\tau}^T \right\}. \\
  \end{aligned}
\end{displaymath}

Then, we introduce the standard trace operators that are commonly used
in  DG framework. Let $\bm{v}$ be a vector- or tensor-valued function
and $e$ be an interior face shared by elements $K^+$ and $K^-$ with
the unit outward norm $\un^+$ and $\un^-$ corresponding to $\partial
K^+$ and $\partial K^-$, respectively. The average operator
$\aver{\cdot}$ and the jump operator $\jump{\cdot}$ are defined as
follows:
\begin{displaymath}
  \aver{\bm{v}} = \frac{1}{2}\left( \bm{v}|_{K^+} + \bm{v}|_{K^-}
  \right), \quad \jump{\bm{v}} = \begin{cases}
    \bm{v}|_{K^+} \otimes \un^+ + \bm{v}|_{K^-} \otimes \un^-, \quad
    &\text{for vector }\bm{v}, \\  
    \bm{v}|_{K^+} \cdot \un^+ + \bm{v}|_{K^-} \cdot \un^-, \quad
    &\text{for tensor }\bm{v}, \\  
  \end{cases}
\end{displaymath}
and in the case $e \in \MEh^b$, $\aver{\cdot}$ and $\jump{\cdot}$ are
modified as 
\begin{displaymath}
  \aver{\bm{v}} = \bm{v},  \quad \jump{\bm{v}} = \begin{cases}
    \bm{v}|_{K} \otimes \un , \quad
    &\text{for vector }\bm{v}, \\  
    \bm{v}|_{K} \cdot \un , \quad 
    &\text{for tensor }\bm{v}, \\  
  \end{cases}
\end{displaymath}
where $\un$ is the unit outward normal on $e$.

Now let us define the following least squares functional for the
problem \eqref{eq:problem}: 
\begin{equation}
  \begin{aligned}
    J_h(\bm{\sigma}, \bm{u}) \triangleq & \sum_{ K \in \MTh}\left(
    \|\mc{A}\bm{\sigma} -\bm{\varepsilon}( \bm{u}) \|_{L^2(K)}^2 +
    \|\nabla \cdot \bm{\sigma} + \bm{f} \|_{L^2(K)}^2 \right)  \\
    & + \sum_{e \in \MEh^i} \left( \frac{1}{h_e} \| \jump{\bm{u}}
    \|_{L^2(e)}^2 + \frac{1}{h_e} \|
    \jump{ \bm{\sigma}} \|_{L^2(e)}^2 \right)\\
    &+ \sum_{e \in \Gamma_D} \frac{1}{h_e} \|\bm{u} -
    \bm{g}\|_{L^2(e)}^2 + \sum_{e \in \Gamma_N} \frac{1}{h_e} \|\un
    \cdot \bm{\sigma} - \bm{h} \|_{L^2(e)}^2. \\
  \end{aligned}
  \label{eq:lsfunctional}
\end{equation}
We introduce two approximation spaces based on the reconstructed space
$U_h$: $\bm{\mr{V}}_h$ for the displacement $\bm{u}$ and
$\bm{\Sigma}_h$ for the stress $\bm{\sigma}$ as follows:
\begin{displaymath}
  \bm{\mr{V}}_h = U_h^d, \quad \bm{\Sigma}_h = \left\{ \bm{\tau} \in
  U_h^{d \times d}\ |\ \bm{\tau} = \bm{\tau}^T\ \text{in}\
  \forall K \in \MTh\right\}.
\end{displaymath}
Here we impose the symmetric condition of the stress field in the
solution space with a strong sense.

In this paper, the discontinuous least squares finite element method
reads: {\it find $(\bm{\sigma}_h, \bm{u}_h) \in \bm{\Sigma}_h \times
\bm{\mr{V}}_h$ such that}
\begin{equation}
  J_h(\bm{\sigma}_h, \bm{u}_h) = \inf_{ (\bm{\tau}_h, \bm{v}_h) \in
  \bm{\Sigma}_h \times \bm{\mr{V}}_h} J_h(\bm{\tau}_h, \bm{v}_h).
  \label{eq:infproblem}
\end{equation}
To solve the minimization problem \eqref{eq:infproblem}, one may write
its corresponding variational equation which reads: {\it find
$(\bm{\sigma}_h, \bm{u}_h) \in \bm{\Sigma}_h \times \bm{\mr{V}}_h$
such that
\begin{displaymath}
  a_h(\bm{\sigma}_h, \bm{u}_h; \bm{\tau}_h, \bm{v}_h) =
  l_h(\bm{\tau}_h, \bm{v}_h), \quad \forall (\bm{\tau}_h, \bm{v}_h)
  \in  \bm{\Sigma}_h \times \bm{\mr{V}}_h,
\end{displaymath}
where the bilinear form $a_h(\cdot; \cdot)$ and the linear form
$l_h(\cdot)$ are defined as
\begin{equation}
  \begin{aligned}
    a_h(\bm{\sigma}_h, \bm{u}_h; \bm{\tau}_h, \bm{v}_h) &= \sum_{K \in
    \MTh} \int_K \left( \mc A \bm{\sigma}_h -
    \bm{\varepsilon}(\bm{u}_h)\right) : \left( \mc A \bm{\tau}_h -
    \bm{\varepsilon}(\bm{v}_h) \right) \d{x} \\
    + \sum_{K \in \MTh} \int_K
    (\nabla \cdot \bm{\sigma}_h) \cdot (\nabla \cdot \bm{\tau}_h)
    \d{x}& + \sum_{e \in \MEh^i \cup \Gamma_D} \int_e \frac{1}{h_e}
    \jump{\bm{u}_h} : \jump{\bm{v}_h} \d{s} + \sum_{e \in \MEh^i \cup
    \Gamma_N} \int_e \frac{1}{h_e} \jump{\bm{\sigma}_h} \cdot
    \jump{\bm{\tau}_h} \d{s},  \\
  \end{aligned}
  \label{eq:bilinear}
\end{equation}
and
\begin{displaymath}
  \begin{aligned}
    l_h(\bm{\tau}_h, \bm{v}_h) = - \sum_{K \in \MTh} \int_K (\nabla
    \cdot \bm{\tau}_h) \cdot \bm{f} \d{x} + \sum_{e \in \Gamma_D}
    \int_e \frac{1}{h_e} \bm{v}_h \cdot \bm{g} \d{s} + \sum_{e \in
    \Gamma_N} \int_e \frac{1}{h_e} (\bm{\tau}_h \cdot \un) \cdot
    \bm{h} \d{s}.
  \end{aligned}
\end{displaymath}}
Below we would concentrate on the uniform continuity and ellipticity
of the bilinear form $a_h(\cdot, \cdot)$. To do so, we first introduce
two energy norms $\snorm{\cdot}$ and $\unorm{\cdot}$:
\begin{displaymath}
  \begin{aligned}
    \snorm{\bm{\tau}_h}^2 \triangleq &\sum_{K \in \MTh} \left( \|
    \bm{\tau}_h \|_{L^2(K)}^2 + \| \nabla \cdot \bm{\tau}_h
    \|_{L^2(K)}^2 \right) + \sum_{e \in \MEh^i \cup \Gamma_N}
    \frac{1}{h_e} \| \jump{\bm{\tau}_h} \|_{L^2(e)}^2, \quad \forall
    \bm{\tau}_h \in H^1(\MTh)^{d \times d}, \\
    \unorm{\bm{v}_h}^2  \triangleq &\sum_{K \in \MTh} \|
    \bm{\varepsilon}(\bm{v}_h) \|_{L^2(L)}^2 + \sum_{e \in \MEh^i \cup
    \Gamma_D} \frac{1}{h_e} \| \jump{\bm{v}_h} \|_{L^2(e)}^2, \quad
    \forall \in H^1(\MTh)^d. \\
  \end{aligned}
\end{displaymath}
Obviously, $\snorm{\cdot}$ actually defines a norm on the space
$H^1(\MTh)^{d \times d}$. The following lemma ensures
$\unorm{\cdot}$ is indeed a norm on the space $H^1(\MTh)^d$.
\begin{lemma}
  For any function $\bm{v}_h \in H^1(\MTh)^d$, the following Korn's
  inequality holds true
  \begin{equation}
    \|\bm{v}_h\|_{H^1(\MTh)}^2 \leq C \left( \sum_{K \in \MTh} \left\|
    \bm{\varepsilon}(\bm{v}_h) \right\|_{L^2(K)}^2 + \sum_{e \in
    \MEh^i \cup \Gamma_D} \frac{1}{h_e} \| \jump{\bm{v}_h}
    \|_{L^2(e)}^2 \right). 
    \label{eq:Korn}
  \end{equation}
  \label{le:Korn}
\end{lemma}
\begin{proof}
  The proof could be found in \cite{Brenner2004korn}.
\end{proof}
Then we state the continuity result of the bilinear form with respect
to the norms $\snorm{\cdot}$ and $\unorm{\cdot}$.
\begin{lemma}
  For the bilinear form $a_h(\cdot; \cdot)$, the following estimates
  holds:
  \begin{equation}
    |a_h(\bm{\sigma}_h, \bm{u}_h; \bm{\tau}_h, \bm{v}_h)| \leq C
    \left( \snorm{\bm{\sigma}_h}^2 + \unorm{\bm{u}_h}^2
    \right)^{\frac{1}{2}} \left( \snorm{\bm{\tau}_h}^2 +
    \unorm{\bm{v}_h}^2 \right)^{\frac{1}{2}},
    \label{eq:boundedness}
  \end{equation}
  for any $(\bm{\sigma}_h, \bm{u}_h), (\bm{\tau}_h, \bm{v}_h) \in
  H^1(\MTh)^{d \times d} \times H^1(\MTh)^d$.
  \label{le:boundedness}
\end{lemma}
\begin{proof}
  We only need to bound the term $\|\mc A \bm{\sigma}_h\|$:
  \begin{displaymath}
    \begin{aligned}
      \|\mc A \bm{\sigma}_h\|_{L^2(K)}^2 & = \left(
      \frac{1}{2\mu} \right)^2 \left( \|\bm{\sigma}_h \|_{L^2(K)}^2 -
      \frac{2\lambda}{d\lambda + 2\mu} \| \tr \bm{\sigma}_h
      \|_{L^2(K)}^2 + d \left( \frac{\lambda}{d \lambda + 2\mu}
      \right)^2 \| \tr \bm{\sigma}_h \|^2 \right) \\
      & \leq \frac{1}{4\mu^2} \| \bm{\sigma}_h\|_{L^2(K)}^2, \\
    \end{aligned}
  \end{displaymath}
  which directly gives us 
  \begin{equation}
    \|\mc A \bm{\sigma}_h \|_{L^2(\Omega)} \leq C \|\bm{\sigma}_h
    \|_{L^2(\Omega)}, \quad \|\mc A \bm{\tau}_h \|_{L^2(\Omega)} \leq
    C \|\bm{\tau}_h \|_{L^2(\Omega)}. 
    \label{eq:AL2upper}
  \end{equation}
  Applying the Cauchy-Schwarz inequality to \eqref{eq:bilinear} and
  using \eqref{eq:AL2upper} could yield the estimate
  \eqref{eq:boundedness}, which completes the proof.
\end{proof}
In order to prove the coercivity of the bilinear form $a_h(\cdot,
\cdot)$, we may require the following lemmas. 
\begin{lemma}
  For any $\bm{\tau} \in L^2(\Omega)^{d \times d}$, there exists a
  constant $C$ such that 
  \begin{equation}
    \|\bm{\tau}\|_{L^2(\Omega)} \leq C \left( (\mc{A} \bm{\tau},
    \bm{\tau}) + \|\nabla \cdot \bm{\tau} \|_{-1, D}^2  \right)^{
    \frac{1}{2}}.
    \label{eq:L2upper}
  \end{equation}
  \label{le:L2upper}
\end{lemma}
\begin{proof}
  We split the $\| \bm{\tau} \|_{L^2(\Omega)}$ into two parts:
  \begin{displaymath}
    \| \bm{\tau} \|_{L^2(\Omega)}^2 = 2\mu(\mc{A} \bm{\tau},
    \bm{\tau}) + \frac{\lambda}{d\lambda + 2\mu} \| \tr \bm{\tau}
    \|_{L^2(\Omega)}^2 \leq C(\mc{A} \bm{\tau}, \bm{\tau}) +
    \frac{1}{d} \| \tr \bm{\tau} \|_{L^2(\Omega)}^2.
  \end{displaymath}
  Thus, the estimate \eqref{eq:L2upper} demands a bound of $\|
  \tr \bm{\tau}\|_{L^2(\Omega)}^2$. Then we follow the idea in
  \cite{Cai2003first} to apply the Helmholtz decomposition and here we
  prove for the case $d = 2$. Let $\bm{q} \in H_D^1(\Omega)$ be the
  solution of the problem
  \begin{displaymath}
    \nabla \cdot (\mc A^{-1} \nabla \bm{q}) = \nabla \cdot \bm{\tau}
    \quad \text{in } \Omega, \quad \bm{q} = 0 \quad \text{on }
    \Gamma_D, \quad \un \cdot (\mc A^{-1} \nabla \bm{q}) = 0 \quad
    \text{on } \Gamma_N,
  \end{displaymath}
  whose weak formulation is that $\bm{q}$ is the only solution of
  \begin{displaymath}
    \lambda(\nabla \cdot \bm{q}, \nabla \cdot \bm{\xi}) + 2\mu (\nabla
    \bm{q}, \nabla \bm{\xi}) = (\bm{\tau}, \nabla \bm{\xi}), \quad
    \forall \bm{\xi} \in H^1_D(\Omega)^d.
  \end{displaymath}
  Taking $\bm{\xi} = \bm{q}$, together with the Poincare inequality
  $ \|q\|_{H^1(\Omega)} \leq C |\nabla \bm{q}|_{L^2(\Omega)}$,
  directly yields
  \begin{displaymath}
    \lambda \|\nabla \cdot \bm{q} \|_{L^2(\Omega)}^2 + C_1 \|\bm{q}
    \|_{H^1(\Omega)}^2 \leq C_2 \|\nabla \cdot \bm{\tau}\|_{-1, D} \|
    \bm{q}\|_{H^1(\Omega)}.
  \end{displaymath}
  Further we apply \cite[Corollary 2.1]{girault1986finite} to obtain
  that 
  \begin{equation}
    \begin{aligned}
      \lambda \|\nabla \cdot \bm{q} \|_{L^2(\Omega)} &\leq C \sup_{
      \bm{v} \in H_D^1(\Omega)^2} \frac{(\lambda \nabla \cdot \bm{q},
      \nabla \cdot \bm{v})}{\|\bm{v}\|_{H^1(\Omega)}} = C \sup_{
      \bm{v} \in H_D^1(\Omega)^2} \frac{(\nabla \cdot \bm{\tau},
      \bm{v}) - (\nabla \bm{q}, \nabla
      \bm{v})}{\|\bm{v}\|_{H^1(\Omega)}}  \\
      &\leq C \left( \|\nabla \cdot \bm{\tau} \|_{-1, D} +
      \|q\|_{H^1(\Omega)} \right) \leq C  \|\nabla \cdot \bm{\tau}
      \|_{-1, D}.
    \end{aligned}
    \label{eq:qL2upper}
  \end{equation}
  For vector-valued function $\bm v=(v_1, v_2)$, we let $\nabla \times
  \bm{v} = \partial_{x_1} v_2 - \partial_{x_2} v_1$ and $\nabla^{\perp}$ be
  the formal adjoint of the curl:
  \begin{displaymath}
    \nabla^\perp \bm{v} \triangleq \begin{pmatrix}
      \partial_{x_2} v_1 & \partial_{x_2} v_2 \\
      - \partial_{x_1} v_1 & - \partial_{x_1} v_2 \\
    \end{pmatrix}.
  \end{displaymath}
  As $\bm{\tau} - \mc{A}^{-1} \nabla \bm{q}$ is divergence-free,
  \cite[Theorem 3.1]{girault1986finite} implies a decomposition that
  there exists a unique solution $\bm{\phi} \in H_N^1(\Omega)^2$ of 
  \begin{displaymath}
    \nabla \times (\mc{A} \nabla^\perp \bm{\phi}) = \nabla
    \times (\mc{A} \bm{\tau}) \text{ in } \Omega, \  \un
    \times(\mc{A} \nabla^\perp \bm{\phi}) = \un \times (\mc{A}^{-1}
    \bm{\tau}), \text{ on } \Gamma_D, \ \bm{\phi} = 0, \text{ on }
    \Gamma_N,
  \end{displaymath}
  such that 
  \begin{displaymath}
    \bm{\tau} = \mc{A}^{-1} \nabla \bm{q} + \nabla \times \bm{\phi}.
  \end{displaymath}
  From the definition of $\mc{A}$ and combining with the regularity of
  $\bm{\phi}$, we observe that
  \begin{displaymath}
    \begin{aligned}
      \frac{1}{2\mu} \left( \|\nabla^\perp \bm{\phi}
      \|_{L^2(\Omega)}^2 - \frac{\lambda}{2(\lambda + \mu)} \| \nabla
      \times \bm{\phi} \|_{L^2(\Omega)}^2 \right) = (\mc{A}
      \nabla^\perp \bm{\phi}, \nabla^\perp \bm{\phi}) \leq C (\mc{A}
      \bm{\tau}, \bm{\tau}).
    \end{aligned}
  \end{displaymath}
  Applying the trace operator brings us that 
  \begin{displaymath}
    \tr \bm{\tau} = 2(\lambda + \mu) \nabla \cdot \bm{q} - \nabla
    \times \bm{\phi}.
  \end{displaymath}
  From the decomposition, $\nabla^\perp \bm{\phi}$ is divergence free
  which satisfies that
  \begin{displaymath}
    (\nabla^\perp \bm{\phi}, \nabla \bm{v}) = 0, \quad \forall \bm{v}
    \in H_D^1(\Omega)^2.
  \end{displaymath}
  Thus, for any $\bm{v} \in H_D^1(\Omega)^2$, one concludes that 
  \begin{displaymath}
    \begin{aligned}
      (\nabla \times \bm{\phi}, \nabla \cdot \bm{v}) &= \left( (\nabla
      \times \bm{\phi})\bm{I}_{2\times 1}, \nabla \bm{v} \right) =
      \left(  (\nabla \times \bm{\phi})\bm{I}_{2\times 1} +
      2\nabla^\perp \bm{\phi}, \nabla \bm{v}\right)\\
      & \leq C\left( \|\nabla^\perp \bm{\phi} \|_{L^2(\Omega)}^2 -
      \frac{1}{2} \| \nabla \times \bm{\phi} \|_{L^2(\Omega)}^2
      \right)^{\frac{1}{2}} \| \nabla \bm{v} \|_{L^2(\Omega)},
    \end{aligned}
  \end{displaymath}
  where $\bm{I}_{2\times1} = (1, 1)^T$. We again apply the estimate
  \eqref{eq:qL2upper} to get that 
  \begin{displaymath}
    \begin{aligned}
      \|\nabla \times \bm{\phi} \|_{L^2(\Omega)} \leq C \sup_{\bm{v}
      \in H_D^1(\Omega)^2} \frac{(\nabla \times \bm{\phi}, \nabla
      \times \bm{v})}{\|\bm{v}\|_{H^1(\Omega)}} \leq  C\left(
      \|\nabla^\perp \bm{\phi} \|_{L^2(\Omega)}^2 - \frac{1}{2} \|
      \nabla \times \bm{\phi} \|_{L^2(\Omega)}^2
      \right)^{\frac{1}{2}}. 
    \end{aligned}
  \end{displaymath}
  Collecting above estimates could directly lead to a bound of $\tr
  \bm{\tau}$ in $L^2$ norm, which completes the proof in the case $d =
  2$. Besides, the proof could be extended to the case $d=3$ without
  any difficulty.
\end{proof}
\begin{lemma}
  For any $\bm{\tau}_h \in \bm{\Sigma}_h$, there exists a constant $C$
  such that 
  \begin{equation}
    \|\nabla \cdot \bm{\tau}_h\|_{-1, D} \leq C \left( \sum_{K \in
    \MTh} \| \nabla \cdot \bm{\tau}_h\|_{L^2(K)}^2 + \sum_{e \in
    \MEh^i \in \Gamma_N} \frac{1}{h_e} \| \jump{\bm{\tau}_h}
    \|_{L^2(e)}^2 \right)^{\frac{1}{2}}.
    \label{eq:nnormupper}
  \end{equation}
  \label{le:nnormupper}
\end{lemma}
\begin{proof}
  From the definition \eqref{eq:nnormdef}, we clearly have that 
  \begin{displaymath}
    \|\nabla \cdot \bm{\tau}_h \|_{-1, D} = \sup_{0 \neq \bm{\psi} \in
    H_D^1(\Omega)^d} \frac{(\bm{\tau}_h, \nabla \bm{\psi}) }{ \|
    \bm{\psi} \|_{H^1(\Omega)}}, 
  \end{displaymath}
  and 
  \begin{displaymath}
    \begin{aligned}
      (\bm{\tau}_h, \nabla \bm{\psi}) &= \sum_{K \in \MTh} \int_K
      \bm{\tau}_h : \nabla \bm{\psi} \d{x} = \sum_{K \in \MTh} \left(
      \int_{\partial K} (\bm{\tau}_h \cdot \un) \cdot \bm{\psi} \d{s}
      - \int_K (\nabla \cdot \bm{\tau}_h) \cdot \bm{\psi} \d{x}
      \right) \\
      &= \sum_{e \in \MEh} \int_e \jump{\bm{\tau}_h} \cdot \bm{\psi}
      \d{s} - \sum_{K \in \MTh} \int_K (\nabla \cdot \bm{\tau}_h)
      \cdot \bm{\psi} \d{x} \\
      & \leq C \left( \sum_{K \in \MTh} \| \nabla \cdot
      \bm{\tau}_h\|_{L^2(K)}^2 + \sum_{e \in \MEh^i \in \Gamma_N}
      \frac{1}{h_e} \| \jump{\bm{\tau}_h} \|_{L^2(e)}^2  \right)^{
      \frac{1}{2}} \| \bm{\psi} \|_{H^1(\Omega)}, \\
    \end{aligned}
  \end{displaymath}
  where the last inequality follows from Cauchy-Schwarz inequality and
  the trace inequality \eqref{eq:traceinequality}, which completes the
  proof.
\end{proof}
Now we are ready to state that the bilinear form $a_h(\cdot; \cdot)$
is coercive with respect to the energy norms.
\begin{lemma}
  For the bilinear form $a_h(\cdot; \cdot)$, there exists a constant
  $C$ such that
  \begin{equation}
    a_h(\bm{\sigma}_h, \bm{u}_h; \bm{\sigma}_h, \bm{u}_h) \geq C
    \left( \snorm{\bm{\sigma}_h}^2 + \unorm{\bm{u}_h}^2 \right), 
    \label{eq:coercivity}
  \end{equation}
  for any $(\bm{\sigma}_h, \bm{u}_h) \in \bm{\Sigma}_h \times
  \bm{\mr{V}}_h$.
  \label{le:coercivity}
\end{lemma}
\begin{proof}
  The key point is to prove that 
  \begin{displaymath}
    \|\bm{\sigma}_h\|_{L^2(\Omega)} \leq C a_h^{\frac{1}{2}}
    (\bm{\sigma}_h, \bm{u}_h; \bm{\sigma}_h, \bm{u}_h),
  \end{displaymath}
  where $ a_h^{\frac{1}{2}} (\bm{\sigma}_h, \bm{u}_h; \bm{\sigma}_h,
  \bm{u}_h) $ denotes the square root of $ a_h (\bm{\sigma}_h,
  \bm{u}_h; \bm{\sigma}_h, \bm{u}_h)$.

  By using Cauchy-Schwarz inequality, we could observe that
  \begin{equation}
    \begin{aligned}
      (\mc A \bm{\sigma}_h, \bm{\sigma}_h) &= \sum_{K \in \MTh} \int_K
      (\mc A \bm{\sigma}_h - \bm{\varepsilon}(\bm{u}_h)) :
      \bm{\sigma}_h \d{x} + \sum_{K \in \MTh} \int_K
      \bm{\varepsilon}(\bm{u}_h) : \bm{\sigma}_h \d{x} \\
      & \leq C a_h^{ \frac{1}{2}}(\bm{\sigma}_h, \bm{u}_h;
      \bm{\sigma}_h, \bm{u}_h) \| \bm{\sigma}_h\|_{L^2(\Omega)} +
      \sum_{K \in \MTh} \int_K \bm{\varepsilon}(\bm{u}_h) :
      \bm{\sigma}_h \d{x}. \\
    \end{aligned}
    \label{eq:Ass}
  \end{equation}
  We then apply the element-wise integration by parts to get
  \begin{displaymath}
    \begin{aligned}
      \sum_{K \in \MTh} \int_K \bm{\varepsilon}&(\bm{u}_h) :
      \bm{\sigma}_h \d{x} = \sum_{K \in \MTh} \int_K \nabla \bm{u}_h
      : \bm{\sigma}_h \d{x} = \sum_{K \in \MTh} \left( \int_{\partial
      K} (\bm{\sigma}_h \cdot \un) \cdot \bm{u}_h \d{s} - \int_K
      (\nabla \cdot \bm{\sigma}_h) \cdot \bm{u}_h \d{x} \right) \\
      &= \sum_{e \in \MEh^i \cup \Gamma_D} \int_e \aver{\bm{\sigma}_h}
      : \jump{\bm{u}_h} \d{s} + \sum_{e \in \MEh^i \cup \Gamma_N}
      \int_e \aver{\bm{u}_h} \cdot \jump{\bm{\sigma}_h} \d{s} -
      \sum_{K \in \MTh} \int_K (\nabla \cdot \bm{\sigma}_h) \cdot
      \bm{u}_h \d{x}.  \\
    \end{aligned}
  \end{displaymath}
  By Cauchy-Schwarz inequality, we immediately obtain
  \begin{displaymath}
    \begin{aligned}
       \sum_{e \in \MEh^i \cup \Gamma_D}\int_e \aver{\bm{\sigma}_h}
       : \jump{\bm{u}_h} \d{s} \leq C \left(  \sum_{e \in \MEh^i
       \cup \Gamma_D} \int_e h_e \|\aver{\bm{\sigma}_h}\|^2_{L^2(e)}
       \d{s} \right)^{\frac{1}{2}} \left(  \sum_{e \in \MEh^i \cup
       \Gamma_D}\int_e \frac{1}{h_e} \| \jump{\bm{u}_h} \|^2_{L^2(e)}
       \d{s} \right)^{\frac{1}{2}}.
    \end{aligned}
  \end{displaymath}
  For $e \subset \partial K$, we employ the trace inequality
  \eqref{eq:traceinequality} and the inverse inequality
  \eqref{eq:inverseinequality} to find that 
  \begin{displaymath}
    \begin{aligned}
      h_K \|\bm{\sigma}_h\|_{L^2(\partial K)}^2 \leq C\left(
      \|\bm{\sigma}_h\|_{L^2(K)}^2 + h_K^2 \|\nabla
      \bm{\sigma}_h\|_{L^2(K)}^2 \right) \leq C
      \|\bm{\sigma}_h\|_{L^2(K)}^2.
    \end{aligned}
  \end{displaymath}
  Then we conclude that
  \begin{displaymath}
    \sum_{e \in \MEh^i \cup \Gamma_D}\int_e \aver{\bm{\sigma}_h} :
    \jump{\bm{u}_h} \d{s} \leq C a_h^{\frac{1}{2}}
    (\bm{\sigma}_h, \bm{u}_h; \bm{\sigma}_h, \bm{u}_h) \|
    \bm{\sigma}_h \|_{L^2(\Omega)}.
  \end{displaymath}
  Analogously, we could get that 
  \begin{displaymath}
    \sum_{e \in \MEh^i \cup \Gamma_N}\int_e \aver{\bm{u}_h} \cdot
    \jump{\bm{\sigma}_h} \d{s} \leq C a_h^{\frac{1}{2}}
    (\bm{\sigma}_h, \bm{u}_h; \bm{\sigma}_h, \bm{u}_h) \| \bm{u}_h
    \|_{L^2(\Omega)}.
  \end{displaymath}
  Further, by the triangle inequality and Lemma \ref{le:Korn} we have
  \begin{displaymath}
    \begin{aligned}
      \|\bm{u}_h\|_{L^2(\Omega)}^2 &\leq C \left( \sum_{K \in \MTh}
      \int_K \| \bm{\varepsilon}(\bm{u}_h) \|^2 \d{x} + \sum_{e \in
      \MEh^i \cup \Gamma_D} \int_e \frac{1}{h_e} \|
      \jump{\bm{u}_h}\|^2 \d{s} \right) \\
      &\leq C \left( \sum_{K \in \MTh} \int_K \| \mc{A} \bm{\sigma}_h
      -  \bm{\varepsilon}(\bm{u}_h) \|^2 \d{x} + \sum_{e \in \MEh^i
      \cup \Gamma_D} \int_e \frac{1}{h_e} \| \jump{\bm{u}_h}\|^2 \d{s}
       + \|\mc{A} \bm{\sigma}_h\|_{L^2(\Omega)}^2 \right)\\
       &\leq C \left( a_h
       (\bm{\sigma}_h, \bm{u}_h; \bm{\sigma}_h, \bm{u}_h) +
       \|\bm{\sigma}_h \|_{L^2(\Omega)}^2\right).
    \end{aligned}
  \end{displaymath}
  For the last term, we directly observe that
  \begin{displaymath}
    \sum_{K \in \MTh} \int_K (\nabla \cdot \bm{\sigma}_h) \cdot
    \bm{u}_h \d{x} \leq C a_h^{\frac{1}{2}} (\bm{\sigma}_h, \bm{u}_h;
    \bm{\sigma}_h, \bm{u}_h) \| \bm{u}_h \|_{L^2(\Omega)}.
  \end{displaymath}
  Combining above inequalities could yield a bound that
  \begin{displaymath}
     \sum_{K \in \MTh} \int_K \bm{\varepsilon}(\bm{u}_h) :
     \bm{\sigma}_h \d{x} \leq C \left( a_h(\bm{\sigma}_h, \bm{u}_h;
     \bm{\sigma}_h, \bm{u}_h) +  a_h^{\frac{1}{2}}(\bm{\sigma}_h,
     \bm{u}_h; \bm{\sigma}_h, \bm{u}_h)
     \|\bm{\sigma}_h\|_{L^2(\Omega)} \right).
  \end{displaymath}
  We substitute this inequality into \eqref{eq:Ass} and we could know
  that
  \begin{displaymath}
    (\mc A \bm{\sigma}_h, \bm{\sigma}_h) \leq C\left(
    a_h(\bm{\sigma}_h, \bm{u}_h; \bm{\sigma}_h, \bm{u}_h) +
    a_h^{\frac{1}{2}}(\bm{\sigma}_h, \bm{u}_h; \bm{\sigma}_h,
    \bm{u}_h) \|\bm{\sigma}_h\|_{L^2(\Omega)} \right),
  \end{displaymath}
  which, together with the fact $\|\nabla \cdot \bm{\sigma}_h\|_{-1,
  D} \leq C a_h^{\frac{1}{2}}(\bm{\sigma}_h, \bm{u}_h; \bm{\sigma}_h,
  \bm{u}_h)$ and Lemma \ref{le:nnormupper}, implies 
  \begin{displaymath}
    \|\bm{\sigma}_h\|_{L^2(\Omega)} \leq C a_h^{\frac{1}{2}}
    (\bm{\sigma}_h, \bm{u}_h; \bm{\sigma}_h, \bm{u}_h).
  \end{displaymath}
  From the definition of $\snorm{\cdot}$, we conclude that 
  \begin{displaymath}
    a_h(\bm{\sigma}_h, \bm{u}_h; \bm{\sigma}_h, \bm{u}_h) \geq C
    \snorm{ \bm{\sigma}_h}^2.
  \end{displaymath}
  Again we use Lemma \ref{le:L2upper} to estimate $\unorm{\bm{u}_h}$:
  \begin{displaymath}
    \begin{aligned}
      \sum_{K \in \MTh}  \|\bm{\varepsilon}(\bm{u}_h)\|_{L^2(K)}^2
      &\leq C \left( \sum_{K \in \MTh} \int_K \| \mc{A} \bm{\sigma}_h
      - \bm{\varepsilon}(\bm{u}_h) \|^2 \d{x}  + \|\mc{A}
      \bm{\sigma}_h\|_{L^2(\Omega)}^2 \right)\\
      &\leq C \left( a_h (\bm{\sigma}_h, \bm{u}_h; \bm{\sigma}_h,
      \bm{u}_h) + \|\bm{\sigma}_h \|_{L^2(\Omega)}^2\right) \\
      &\leq C a_h (\bm{\sigma}_h, \bm{u}_h; \bm{\sigma}_h, \bm{u}_h). 
      \\
    \end{aligned}
  \end{displaymath}
  Hence, 
  \begin{displaymath}
    a_h(\bm{\sigma}_h, \bm{u}_h; \bm{\sigma}_h, \bm{u}_h) \geq C
    \unorm{\bm{u}_h}^2,
  \end{displaymath}
  which completes the proof.
\end{proof}
Since the bilinear form $a_h(\cdot; \cdot)$ is bounded and coercive,
we have established the existence and uniqueness of the solution to
the minimization problem \eqref{eq:infproblem}. Ultimately, we state
{\it a priori} error estimate of the method proposed in this section: 
\begin{theorem}
  Let $(\bm{\sigma}, \bm{u}) \in H^{m+1}(\Omega)^{\mb{S}, d \times d}
  \times H^{m+1}(\Omega)^d$ be the solution to the problem
  \eqref{eq:problem} and let $(\bm{\sigma}_h, \bm{u}_h)$ be the
  solution to the problem \eqref{eq:infproblem}, there exists a
  constant $C$ such that 
  \begin{equation}
    \snorm{\bm{\sigma} - \bm{\sigma}_h} + \unorm{\bm{u} - \bm{u}_h}
    \leq C h^m \left( \|\bm{\sigma} \|_{H^{m+1}(\Omega)} + \|\bm{u}
    \|_{H^{m+1}(\Omega)} \right).
    \label{eq:pestimate}
  \end{equation}
  \label{th:pestimate}
\end{theorem}
\begin{proof}
  It directly follows from \eqref{eq:bilinear} that for any
  $(\bm{\tau}_h, \bm{v}_h) \in \bm{\Sigma}_h \times \bm{\mr{V}}_h$,
  one has that 
  \begin{displaymath}
    a_h(\bm{\sigma} - \bm{\sigma}_h, \bm{u} - \bm{u}_h; \bm{\tau}_h,
    \bm{v}_h)  = 0.
  \end{displaymath}
  For any $(\bm{\tau}_h, \bm{v}_h) \in \bm{\Sigma}_h \times
  \bm{\mr{V}}_h$, together with \eqref{eq:coercivity} and
  \eqref{eq:boundedness}, we observe that
  \begin{displaymath}
    \begin{aligned}
      \snorm{\bm{\sigma}_h - \bm{\tau}_h}^2 + \unorm{\bm{u}_h -
      \bm{v}_h}^2 &\leq Ca_h(\bm{\sigma}_h - \bm{\tau}_h, \bm{u}_h -
      \bm{v}_h;\bm{\sigma}_h - \bm{\tau}_h, \bm{u}_h - \bm{v}_h) \\
      &= Ca_h(\bm{\sigma} - \bm{\tau}_h, \bm{u} -
      \bm{v}_h;\bm{\sigma}_h - \bm{\tau}_h, \bm{u}_h - \bm{v}_h) \\
      \leq C & \left( \snorm{\bm{\sigma} - \bm{\tau}_h}^2 +
      \unorm{\bm{u} - \bm{v}_h}^2 \right)^{\frac{1}{2}} \left(
      \snorm{\bm{\sigma}_h - \bm{\tau}_h}^2 + \unorm{\bm{u}_h -
      \bm{v}_h}^2 \right)^{\frac{1}{2}}.
    \end{aligned}
  \end{displaymath}
  By the triangle inequality, we get that
  \begin{equation}
     \snorm{\bm{\sigma} - \bm{\sigma}_h} + \unorm{\bm{u} - \bm{u}_h}
     \leq C \inf_{(\bm{\tau}_h, \bm{v}_h) \in \bm{\Sigma}_h \times
     \bm{\mr{V}}_h}\left(  \snorm{\bm{\sigma} - \bm{\tau}_h} +
     \unorm{\bm{u} - \bm{v}_h}\right).
     \label{eq:Cea}
  \end{equation}
  We denote by $(\bm{\sigma}_I, \bm{u}_I)  \in \bm{\Sigma}_h \times
  \bm{\mr{V}}_h$ be the interpolants of $(\bm{\sigma}, \bm{u})$ and we
  only need to estimate the errors of $(\bm{\sigma}_I, \bm{u}_I)$
  under norms $\snorm{\cdot}$ and $\unorm{\cdot}$, respectively. Using
  the trace inequality \eqref{eq:traceinequality} and the
  approximation property \eqref{eq:localapp}, we arrive at 
  \begin{displaymath}
    \begin{aligned}
      \snorm{\bm{\sigma} - \bm{\sigma}_I} \leq Ch^m
      \|\bm{\sigma}\|_{H^{m+1}(\Omega)}, \quad \text{and} \quad
      \unorm{\bm{u} - \bm{u}_I} \leq C h^m \| \bm{u}
      \|_{H^{m+1}(\Omega)}.
    \end{aligned}
  \end{displaymath}
  Substituting the two estimates into \eqref{eq:Cea} implies
  \eqref{eq:pestimate}, which completes the proof.
\end{proof}

%Moreover, we note that the least squares functional could naturally
%serves as a posteriori error estimator. Specifically, we define the
%element estimator $\eta_K(\bm{\sigma}_h, \bm{u}_h)$ as
%\begin{equation}
  %\begin{aligned}
    %\eta_K(\bm{\sigma}_h, \bm{u}_h)^2 \triangleq \|\mc{A}
    %\bm{\sigma}_h - \bm{\varepsilon}(\bm{u}_h) \|_{L^2(K)}^2 + \|
    %\nabla \cdot \bm{\sigma}_h + \bm{f} \|_{L^2(K)}^2 + \frac{1}{h_K}
    %\| \jump{\bm{\sigma}_h} \|_{L^2(\partial K)}^2 + \frac{1}{h_K} \|
    %\jump{\bm{u}_h} \|_{L^2(\partial K)}^2.
  %\end{aligned}
  %\label{eq:etadef}
%\end{equation}
%The adaptive refinement strategies consist in selecting elements with
%the largest values of the estimator $\eta_K$. In each mesh level,
%about $20\%$ of all elements are chosen to be refined regularly.  We
%note that the reconstruction operator proposed in Sec 1 could handle
%the handling nodes without any difficulty and we refer to
%\cite{Li2005refine} for more detail about the implementation of
%adaptive refinement.  

%Clearly, we have the following result:
%\begin{theorem}
  %Let $(\bm{\sigma}_h, \bm{u}_h) \in \bm{\Sigma}_h \times
  %\bm{\mr{V}}_h$ and $\eta_K$ be piecewisely defined as
  %\eqref{eq:etadef}, there exists two constants $C_1$ and $C_2$ such
  %that
  %\begin{equation}
    %C_1\left( \snorm{\bm{\sigma} - \bm{\sigma}_h}^2 + \unorm{\bm{u} -
    %\bm{u}_h}^2 \right) \leq \sum_{K \in \MTh} \leq C_2\left(
    %\snorm{\bm{\sigma} - \bm{\sigma}_h}^2 + \unorm{\bm{u} -
    %\bm{u}_h}^2 \right).
    %\label{eq:estimator}
  %\end{equation}
  %\label{th:estimator}
%\end{theorem}

% vim:spell:tw=70:fo+=Mn:cc=70
\section{Numerical Results}
\label{sec:numericalresults}
In this section, we carry out a series of numerical results in two and
three dimension to exhibit the accuracy and efficiency of the method
proposed in Section \ref{sec:dlsfem}. 
%We set the cardinality $\# S(K)$
%uniformly and a group of reference values of $\# S(K)$ are listed in
%Tab.\ref{tab:numSKd2} and Tab.\ref{tab:numSKd3} for different
%polynomial degree $m$.

\subsection{Convergence order study}
We first demonstrate the convergence behavior to examine the
theoretical prediction and show the flexibility of the proposed
method.

\paragraph{\bf Example 1.} We consider a linear elasticity problem
defined on the unit square $\Omega = [0,1] \times [0,1]$. We let
$\Gamma_N$ be the boundary with $x = 1$ and $\Gamma_D = \partial
\Omega \backslash \Gamma_N$. The exact solution
(see \cite{Grieshaber2015uniformly}) is taken as 
\begin{displaymath}
  \bm{u}(x, y) = \begin{bmatrix}
    \sin(2\pi y)(-1 + \cos(2\pi x)) + \frac{1}{1 + \lambda} \sin(\pi
    x)\sin(\pi y) \\
    \sin(2 \pi x)(1 - \cos(2\pi y)) + \frac{1}{1 + \lambda} \sin(\pi
    x) \sin(\pi y) \\
  \end{bmatrix},
\end{displaymath}
and the stress $\bm{\sigma}$, the source term $\bm{f}$ and the
boundary conditions $\bm{g}$, $\bm{h}$ are taken accordingly. We fix
$\mu = 1$ and $\lambda = 5$ to test the accuracy of the proposed
method.

We solve this problem on a series of triangular meshes (see
Fig.~\ref{fig:triangulation}) with mesh size $h = 1/10, 1/20, \cdots,
1/160$. We set the cardinality $\# S(K)$ uniformly and a group of
reference values of $\# S(K)$ for the case $d=2$ are listed in
Tab.~\ref{tab:numSKd2}. The values of functional $J_h(\bm{\sigma}_h,
\bm{u}_h)$ and errors in $L^2$ norm in the approximation to the exact
solution $(\bm{\sigma}, \bm{u})$ are reported in Fig.~\ref{fig:ex1}.
For fixed $m$, it is clear that the functional $J_h(\bm{\sigma}_h,
\bm{u}_h)$, which is equivalent to the error $\left(
\snorm{\bm{\sigma} - \bm{\sigma}_h}^2 + \unorm{\bm{u} - \bm{u}_h}^2
\right)^{\frac{1}{2}}$, converges to zero at the rate $O(h^m)$ as the
mesh size approaches to zero, and the error $\| \bm{\sigma} -
\bm{\sigma}_h \|_{L^2(\Omega)}$ decreases to zero at the same speed.
For odd $m$, the error $\|\bm{u} - \bm{u}_h\|_{L^2(\Omega)}$ converges
to zero optimally and for even $m$, the convergence order reduces to
$m$. We note that for the case $m = 1$ the convergence rate of the
error $\|\bm{u} - \bm{u}_h\|_{L^2(\Omega)}$ seems less than the
expected value. The predicted convergence rates would recur when the
mesh size $h$ is small enough (see Tab.~\ref{tab:ex1m1}). We refer to
\cite{Cai2003first} for the possible reason and we mainly consider the
case $m \geq 2$ in the rest of this section. In addition, all
numerically detected convergence orders are in agreement with the
theoretical analysis.

\begin{table}
  \centering
  \renewcommand\arraystretch{1.3}
  \begin{tabular}{p{2.0cm}|p{1.3cm}|p{1.3cm}|p{1.3cm}| p{1.3cm}
    |p{1.3cm}}
    \hline\hline 
    $m$ & 1 & 2 & 3 & 4 & 5 \\
    \hline
    $S(K)$ & 4 & 8 & 13 & 19 & 26 \\ 
    \hline\hline
  \end{tabular}
  \caption{$\# S(K)$ for $d=2$.}
  \label{tab:numSKd2}
\end{table}
\begin{figure}
  \centering
  \includegraphics[width=0.4\textwidth]{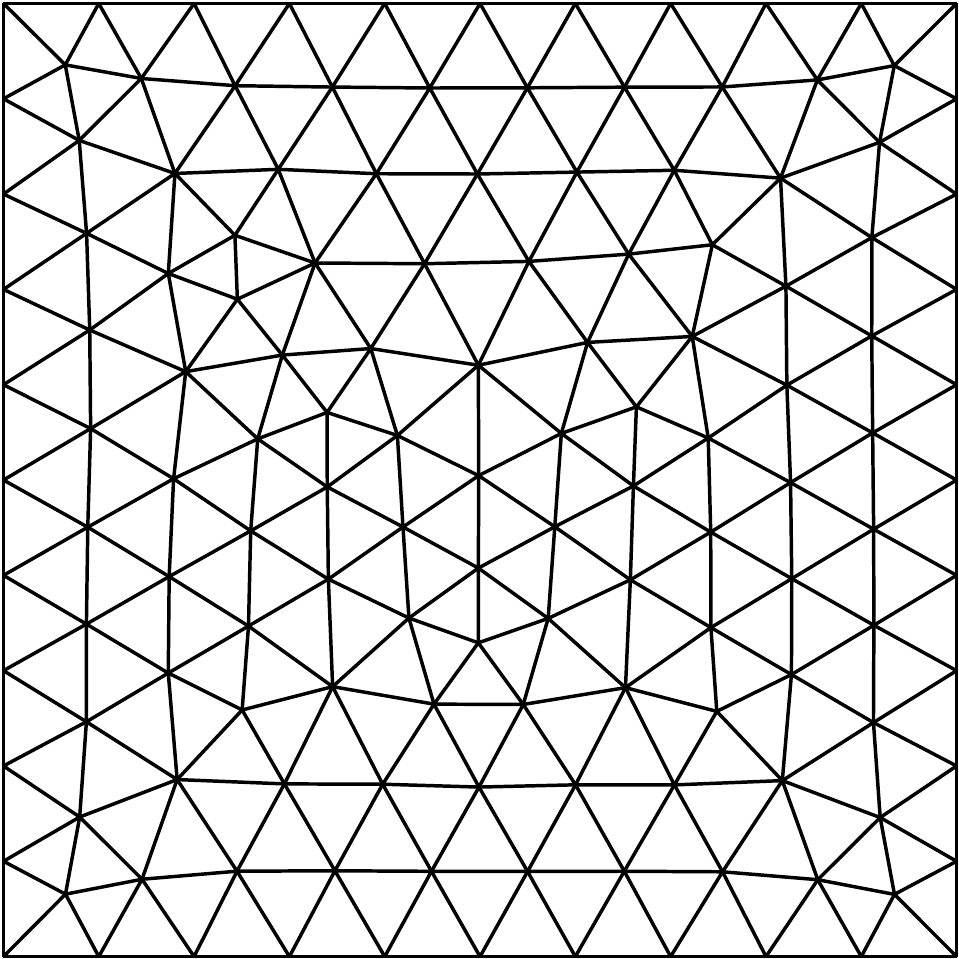}
  \hspace{25pt}
  \includegraphics[width=0.4\textwidth]{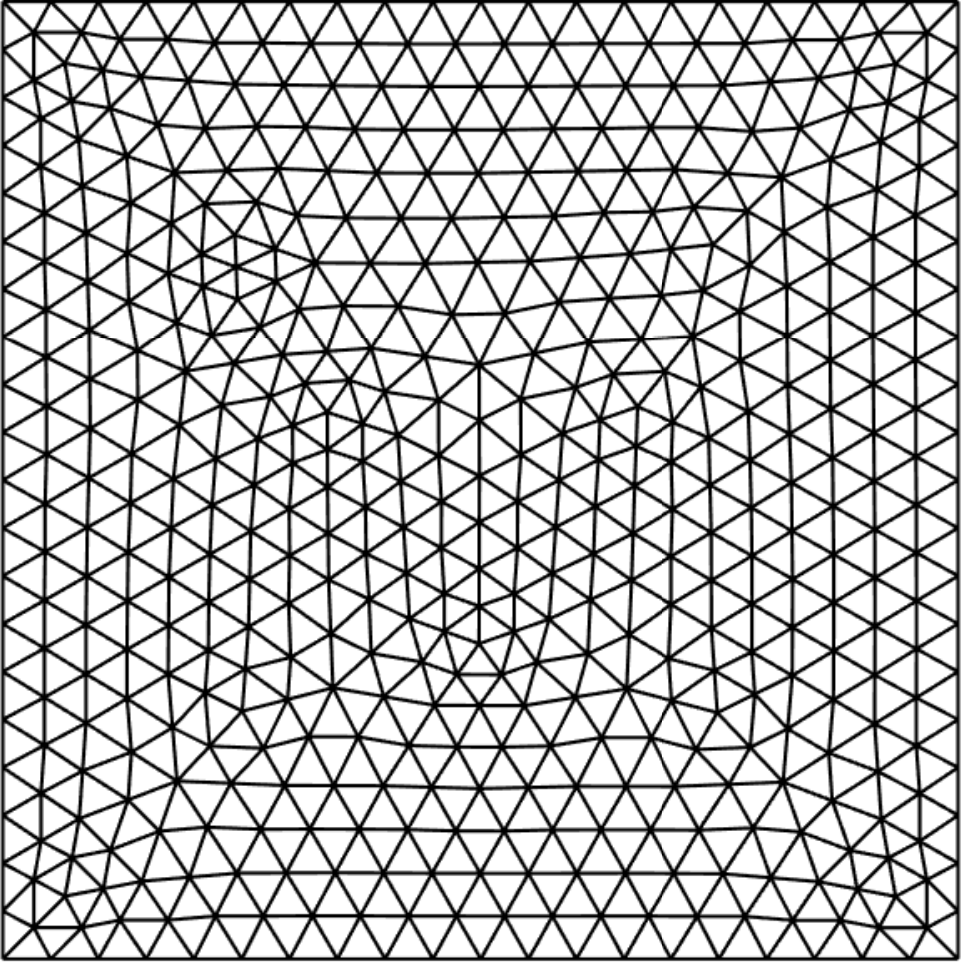}
  \caption{Triangulation with mesh size $h=0.1$ (left)
  / $h=0.05$ (right).}
  \label{fig:triangulation}
\end{figure}

\begin{figure}
  \centering
  \includegraphics[width=0.32\textwidth]{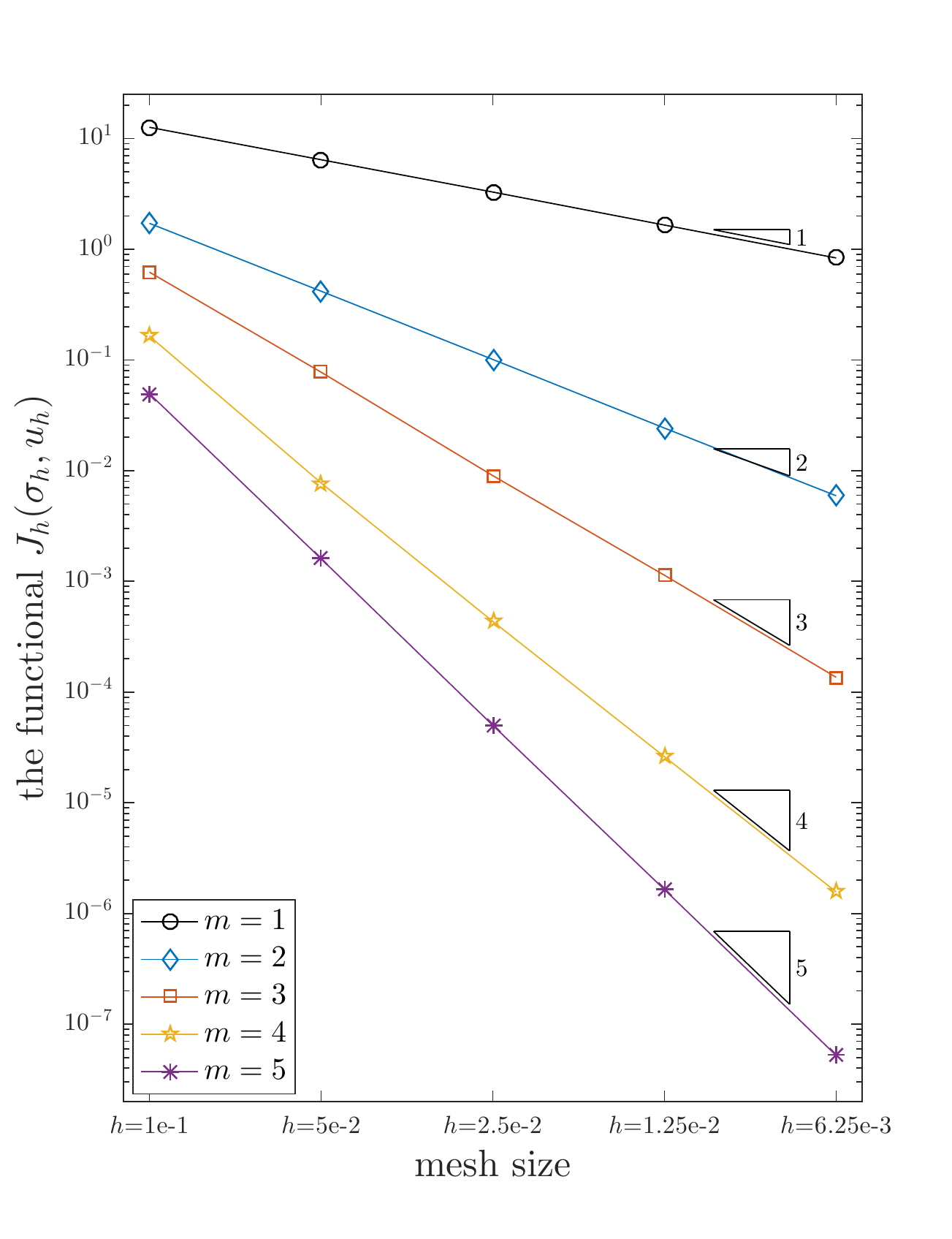}
  \hspace{0pt}
  \includegraphics[width=0.32\textwidth]{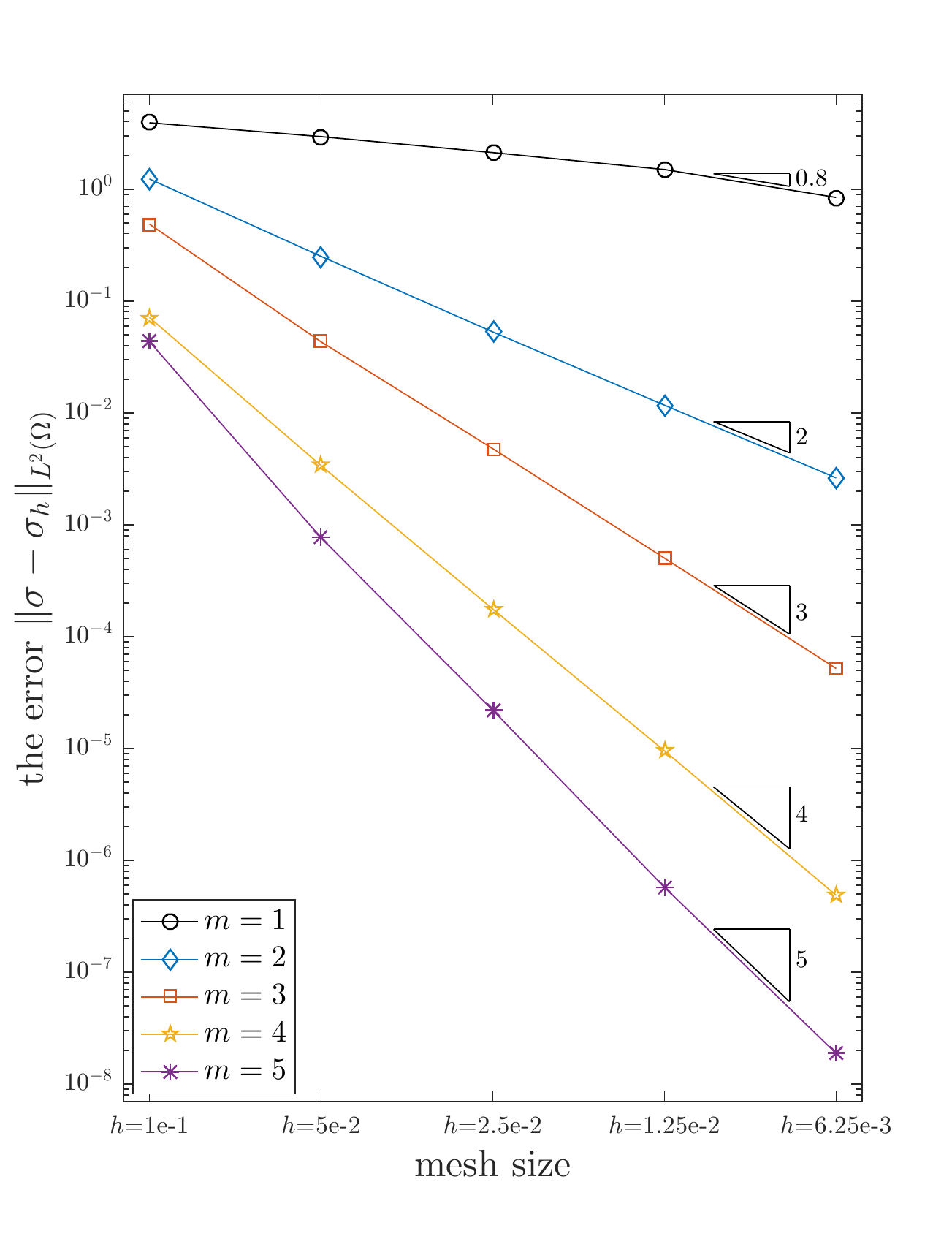}
  \hspace{0pt}
  \includegraphics[width=0.32\textwidth]{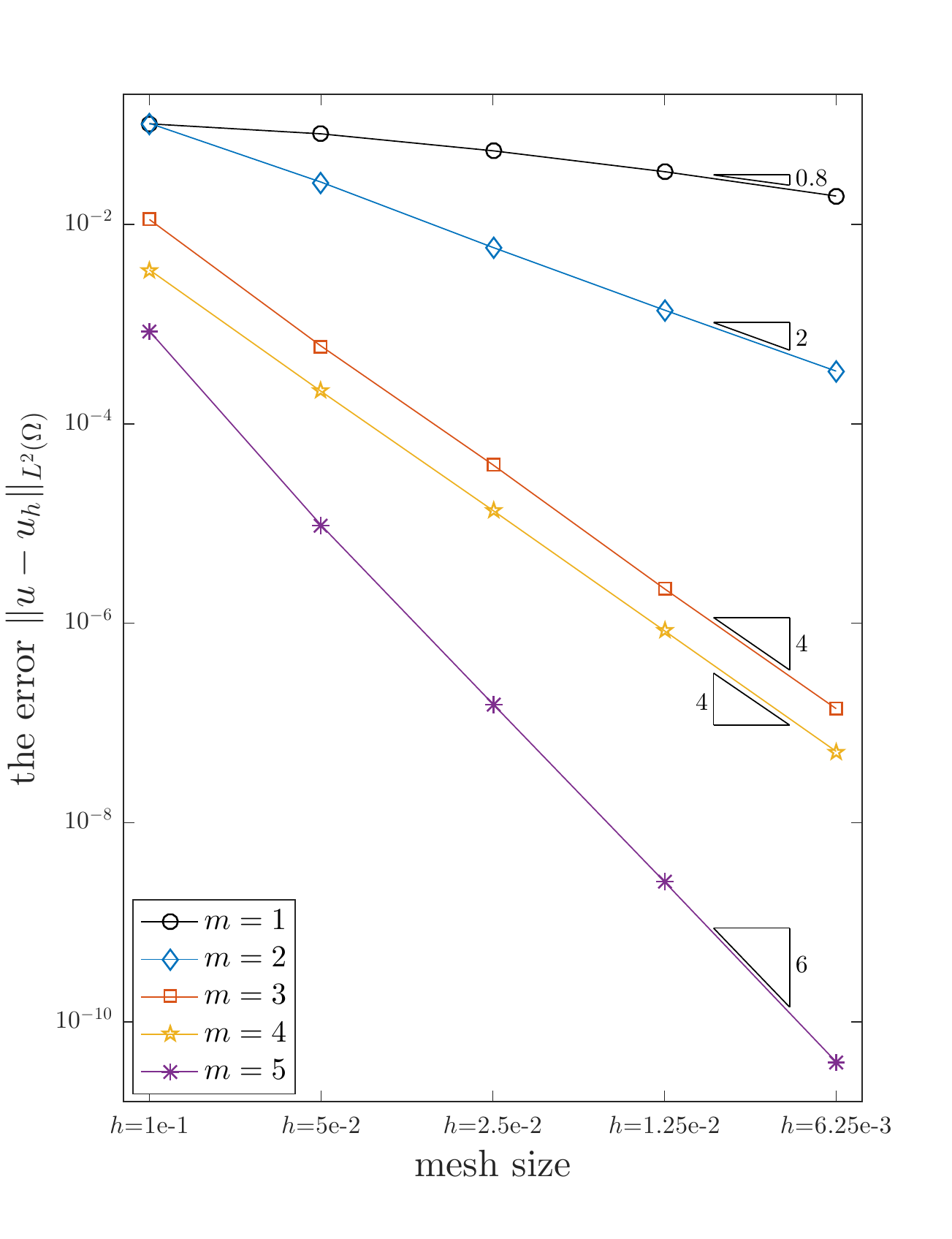}
  \caption{Example 1. The convergence rates of $J_h(\bm{\sigma}_h,
  \bm{u})^{\frac{1}{2}}$ (left) / $\| \bm{\sigma} - \bm{\sigma}_h
  \|_{L^2(\Omega)}$ (mid) / $\| \bm{u} - \bm{u}_h\|_{L^2(\Omega)}$
  (right). }
  \label{fig:ex1}
\end{figure}

\begin{table}
  \centering
  \renewcommand\arraystretch{1.3}
  \begin{tabular}{p{2.3cm} | p{1.3cm} | p{1.3cm} | p{1.3cm} |p{1.3cm} |
    p{1.3cm} |p{1.3cm} | p{1.3cm} } 
    \hline\hline
    mesh level & 1/10 & 1/20 & 1/40 & 1/80 & 1/160 & 1/320 & 1/640 \\
    \hline
  $J_h(\bm{\sigma}_h, \bm{u}_h)^{\frac{1}{2}}$   &
  1.265e+1 & 6.444e-0 & 3.289e-0   & 1.655e-0  &
    8.349e-1  &  4.197e-1  & 2.105e-1  \\
    \hline
    order & - &  0.97  & 0.97  & 0.99   & 0.99  &0.99   & 1.00   \\
    \hline
    $\|\bm{\sigma} - \bm{\sigma}_h \|_{L^2(\Omega)}$  & 3.916e-0 &
    2.943e-0  & 2.121e-0   & 1.496e-0    & 8.417e-1  & 3.971e-1  &
      1.873e-1   \\
    \hline
    order & - & 0.41 & 0.47 & 0.50 & 0.83 & 1.08 & 1.08 \\
    \hline
    $\| \bm{u} - \bm{u}_h\|_{L^2(\Omega)}$ & 1.017e-1 & 8.073e-2  &
    5.461e-2  & 3.371e-2   & 1.920e-2 & 7.781e-3   & 2.679e-3 
    \\
    \hline
    order & - & 0.33 & 0.56 & 0.69 & 0.81 & 1.30 & 1.53 \\
    \hline\hline
  \end{tabular}
  \caption{The convergence rates of $J_h(\bm{\sigma}_h,
  \bm{u}_h)^{\frac{1}{2}}$, $\| \bm{\sigma} -
  \bm{\sigma}_h\|_{L^2(\Omega)}$ and $\| \bm{u} -
  \bm{u}_h\|_{L^2(\Omega)}$ with $m=1$. }
  \label{tab:ex1m1}
\end{table}

Moreover, we exhibit the robustness of the proposed by solving the
problem with $\lambda \rightarrow \infty$. We still take $\bm{u}(x,
y)$ as the exact solution but we select $\lambda=1000, 20000$. The
polynomial degree $m$ is chosen as $2, 3$. The least squares
functional and the errors under $L^2$ norm are gathered in
Tab.~\ref{tab:error2m2df} and Tab.~\ref{tab:error2m3df} for decreasing
mesh size $h$ and different values of the Lam\'e parameter $\lambda$.
Clearly, the numerical solutions produced by our method converge
uniformly as $\lambda$ increases, which confirms the convergence
analysis.

\begin{table}
  \centering
  \renewcommand\arraystretch{1.3}
  \begin{tabular}{p{1.5cm} | p{2.3cm} | p{1.3cm} | p{1.3cm} |p{1.3cm} |
    p{1.3cm} |p{1.3cm} } 
    \hline\hline
    & mesh level & 1/10 & 1/20 & 1/40 & 1/80 & 1/160 \\
    \hline
    \multirow{3}{*}{$ \lambda=5$} & $J_h(\bm{\sigma}_h,
    \bm{u}_h)^{\frac{1}{2}}$ & 1.716e-0 & 4.183e-1 & 1.011e-1 &
    2.426e-2 & 5.935e-3 \\
    \cline{2-7}
    &  $\| \bm{\sigma} - \bm{\sigma}_h\|_{L^2(\Omega)}$ 
    & 1.233e-0 & 2.505e-1 & 5.289e-2 & 1.171e-2 & 2.623e-3 \\
    \cline{2-7}
    & $\| \bm{u} - \bm{u}_h\|_{L^2(\Omega)}$
    & 1.030e-1 & 2.645e-2 & 5.869e-3 & 1.380e-3 & 3.371e-4 \\
    \hline
    \multirow{3}{*}{$ \lambda=1000$} & $J_h(\bm{\sigma}_h,
    \bm{u}_h)^{\frac{1}{2}}$ &
    1.715e-0 & 4.170e-1 & 1.011e-1 & 2.426e-2 & 5.936e-3 \\
    \cline{2-7}
    &  $\| \bm{\sigma} - \bm{\sigma}_h\|_{L^2(\Omega)}$ 
    & 1.192e-0 & 2.935e-1 & 6.845e-2 & 1.389e-2 & 3.025e-3 \\
    \cline{2-7}
    & $\| \bm{u} - \bm{u}_h\|_{L^2(\Omega)}$
    & 1.143e-1 & 2.991e-2 & 6.317e-3 & 1.439e-3 & 3.487e-4 \\
    \hline
    \multirow{3}{*}{$ \lambda=20000$} & $J_h(\bm{\sigma}_h,
    \bm{u}_h)^{\frac{1}{2}}$ &
    1.782e-0 & 4.203e-1 & 9.789e-2 & 2.358e-2 & 5.799e-3 \\
    \cline{2-7}
    &  $\| \bm{\sigma} - \bm{\sigma}_h\|_{L^2(\Omega)}$ 
    & 1.123e-0 & 2.781e-1 & 6.421e-2 & 1.356e-2 & 2.972e-3 \\
    \cline{2-7}
    & $\| \bm{u} - \bm{u}_h\|_{L^2(\Omega)}$
    & 1.083e-1 & 2.946e-2 & 6.192e-3 & 1.391e-3 & 3.366e-4 \\
    \hline\hline
  \end{tabular}
  \caption{$J_h(\bm{\sigma}_h, \bm{u}_h)^{\frac{1}{2}}$, $\|
  \bm{\sigma} - \bm{\sigma}_h\|_{L^2(\Omega)}$ and $\| \bm{u} -
  \bm{u}_h\|_{L^2(\Omega)}$ for different values of $\lambda$ with
  $m=2$. }
  \label{tab:error2m2df}
\end{table}

\begin{table}
  \centering
  \renewcommand\arraystretch{1.3}
  \begin{tabular}{p{1.5cm} | p{2.3cm} | p{1.3cm} | p{1.3cm} |p{1.3cm} |
    p{1.3cm} |p{1.3cm} } 
    \hline\hline
    & mesh level & 1/10 & 1/20 & 1/40 & 1/80 & 1/160 \\
    \hline
    \multirow{3}{*}{$ \lambda=5$} & $J_h(\bm{\sigma}_h,
    \bm{u}_h)^{\frac{1}{2}}$ &
    6.215e-1 & 7.762e-2 & 1.022e-2 & 1.336e-3 & 1.729e-4 \\
    \cline{2-7}
    &  $\| \bm{\sigma} - \bm{\sigma}_h\|_{L^2(\Omega)}$ 
    & 4.866e-1 & 4.313e-2 & 4.789e-3 & 5.033e-4 & 5.193e-5 \\
    \cline{2-7}
    & $\| \bm{u} - \bm{u}_h\|_{L^2(\Omega)}$
    & 1.118e-2 & 6.008e-4 & 3.883e-5 & 2.216e-6 & 1.420e-7 \\
    \hline
    \multirow{3}{*}{$ \lambda=1000$} & $J_h(\bm{\sigma}_h,
    \bm{u}_h)^{\frac{1}{2}}$ &
    6.321e-1 & 7.996e-2 & 1.048e-2 & 1.363e-3 & 1.750e-4 \\
    \cline{2-7}
    &  $\| \bm{\sigma} - \bm{\sigma}_h\|_{L^2(\Omega)}$ 
    & 4.592e-1 & 4.999e-2 & 5.412e-3 & 5.521e-4 & 6.013e-5 \\
    \cline{2-7}
    & $\| \bm{u} - \bm{u}_h\|_{L^2(\Omega)}$
    & 1.353e-2 & 6.991e-4 & 4.638e-5 & 2.612e-6 & 1.610e-7 \\
    \hline
    \multirow{3}{*}{$ \lambda=20000$} & $J_h(\bm{\sigma}_h,
    \bm{u}_h)^{\frac{1}{2}}$ &
    6.198e-1 & 7.723e-2 & 1.019e-2 & 1.333e-3 & 1.718e-4 \\
    \cline{2-7}
    &  $\| \bm{\sigma} - \bm{\sigma}_h\|_{L^2(\Omega)}$ 
    & 4.761e-1 & 5.212e-2 & 5.881e-3 & 6.122e-4 & 6.321e-5 \\
    \cline{2-7}
    & $\| \bm{u} - \bm{u}_h\|_{L^2(\Omega)}$
    & 1.278e-2 & 6.431e-4 & 4.292e-5 & 2.598e-6 & 1.523e-7 \\
    \hline\hline
  \end{tabular}
  \caption{$J_h(\bm{\sigma}_h, \bm{u}_h)^{\frac{1}{2}}$, $\|
  \bm{\sigma} - \bm{\sigma}_h\|_{L^2(\Omega)}$ and $\| \bm{u} -
  \bm{u}_h\|_{L^2(\Omega)}$ for different values of $\lambda$ with
  $m=3$. }
  \label{tab:error2m3df}
\end{table}

\paragraph{\bf Example 2.} In this test, we solve the same problem as
in Example 1 but on a sequence of polygonal meshes. The meshes are
generated by {\tt PolyMesher} \cite{talischi2012polymesher} and
contain very general polygonal elements; see Fig.~\ref{fig:voroni}.
The functional $J_h(\bm{\sigma}_h, \bm{u}_h)^{\frac{1}{2}}$ and the
$L^2$ norm errors in approximation to $(\bm{\sigma}, \bm{u})$ are
displayed in Fig.~\ref{fig:ex2}. Again we observe the values of
$J_h(\bm{\sigma}_h, \bm{u}_h)^{\frac{1}{2}}$ and $\|\bm{\sigma} -
\bm{\sigma}_h\|_{L^2(\Omega)}$ tend to zero at the rate $O(h^m)$. The
convergence order of the error $\|\bm{u}- \bm{u}_h\|_{L^2(\Omega)}$ is
$O(h^{m+1})$ and $O(h^{m})$ for odd and even $m$, respectively. The
numerical results validate our theoretical estimates and highlight the
flexibility of the proposed method.

\begin{figure}
  \centering
  \includegraphics[width=0.4\textwidth]{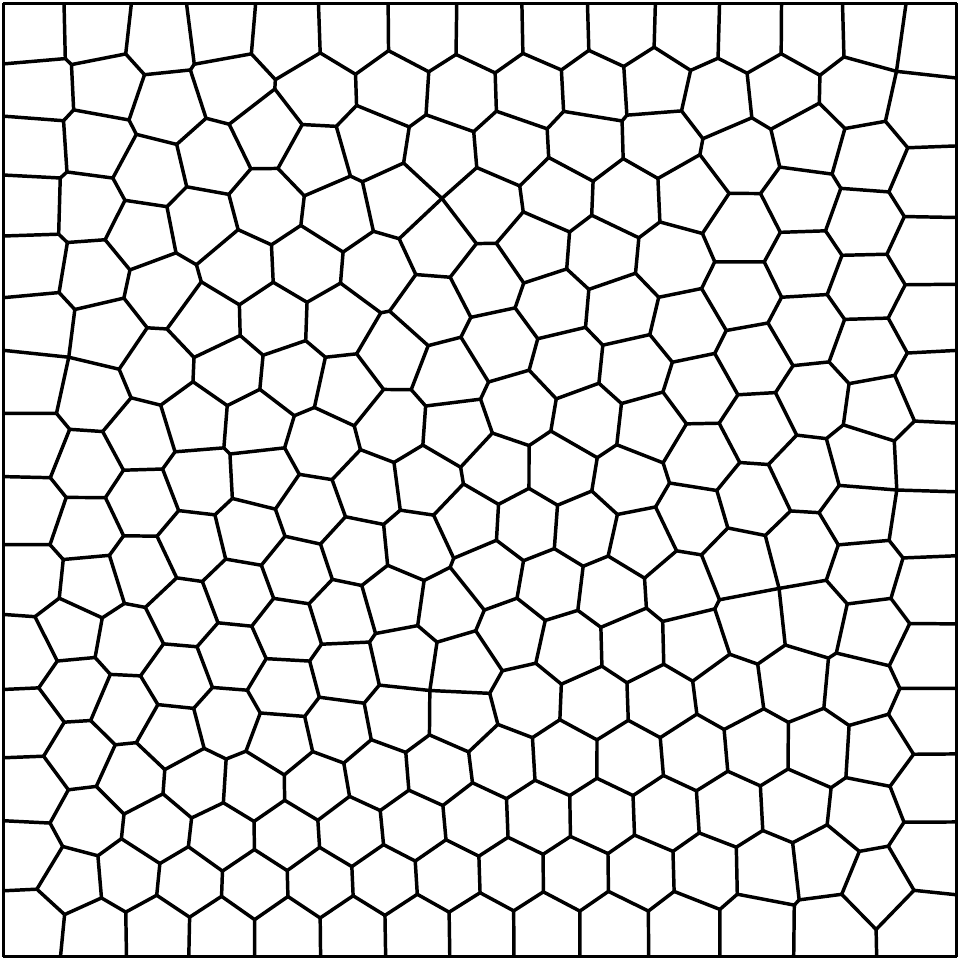}
  \hspace{25pt}
  \includegraphics[width=0.4\textwidth]{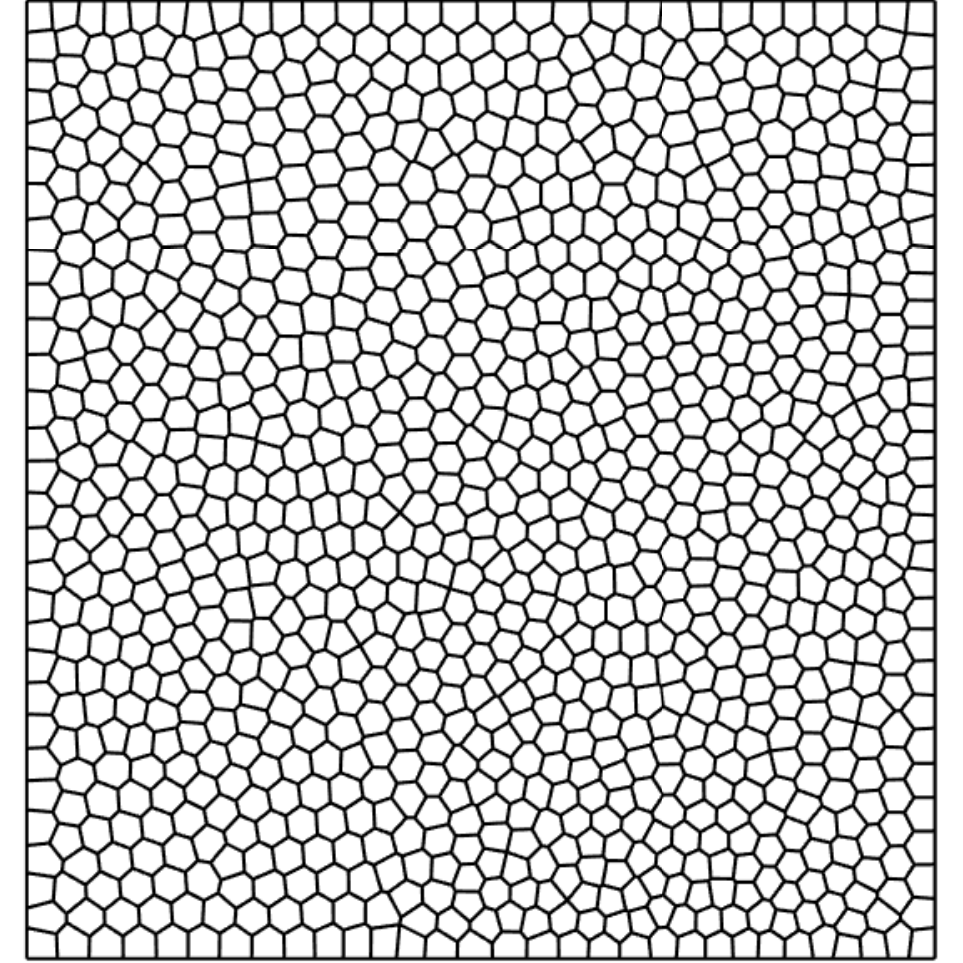}
  \caption{The polygonal meshes with 250 elements (left) / 1000
  elements (right).}
  \label{fig:voroni}
\end{figure}

\begin{figure}
  \centering
  \includegraphics[width=0.32\textwidth]{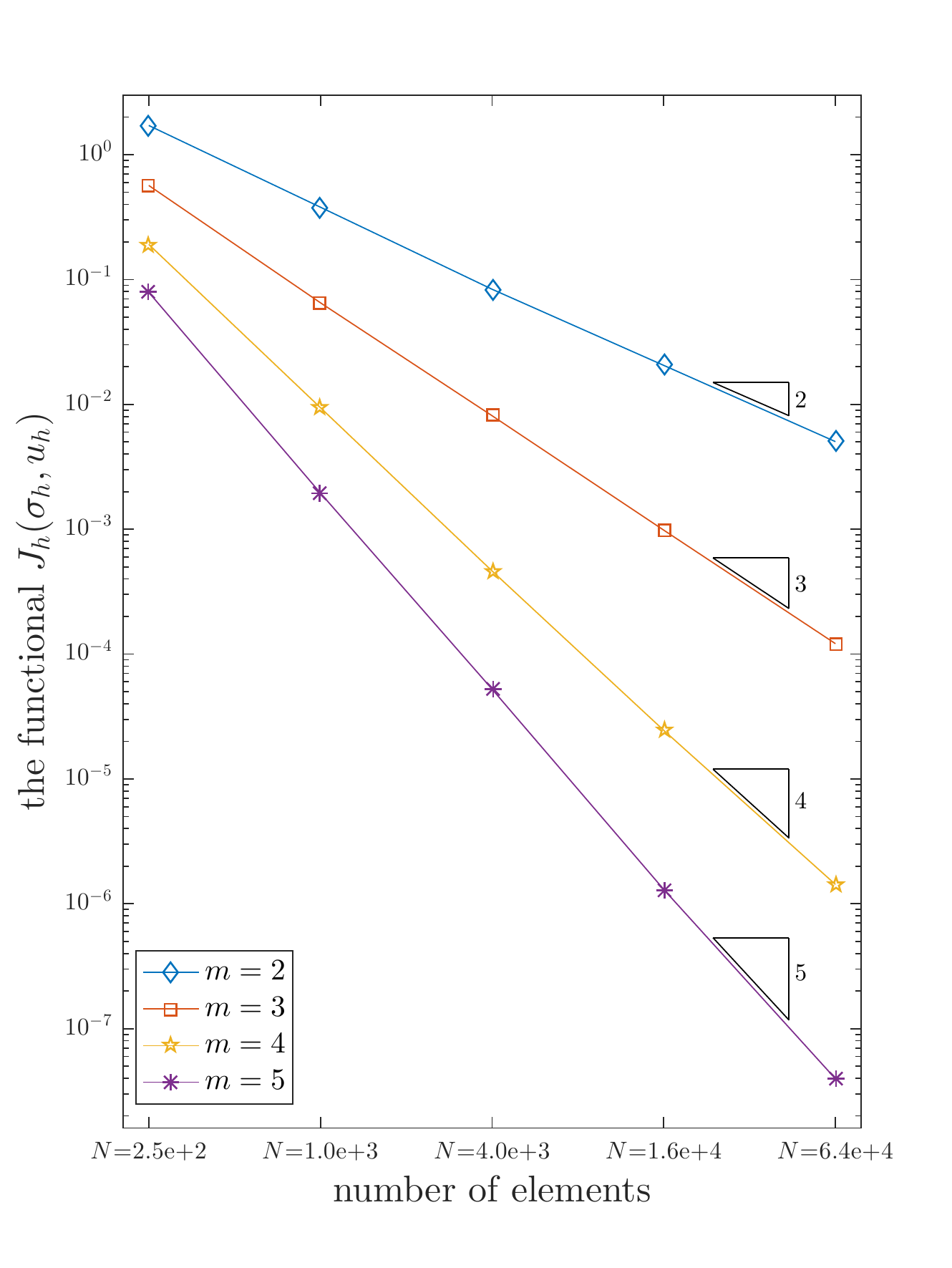}
  \hspace{0pt}
  \includegraphics[width=0.32\textwidth]{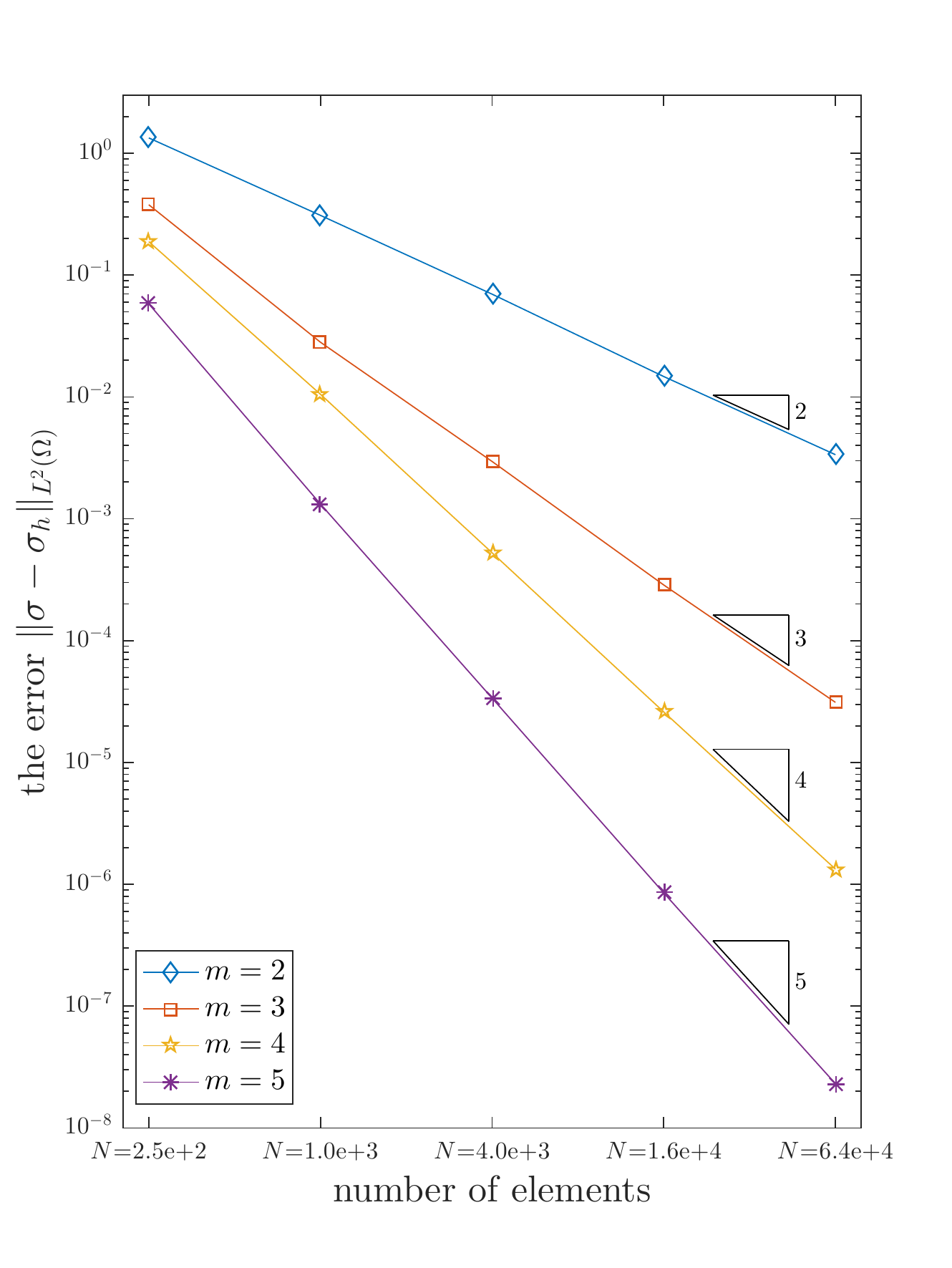}
  \hspace{0pt}
  \includegraphics[width=0.32\textwidth]{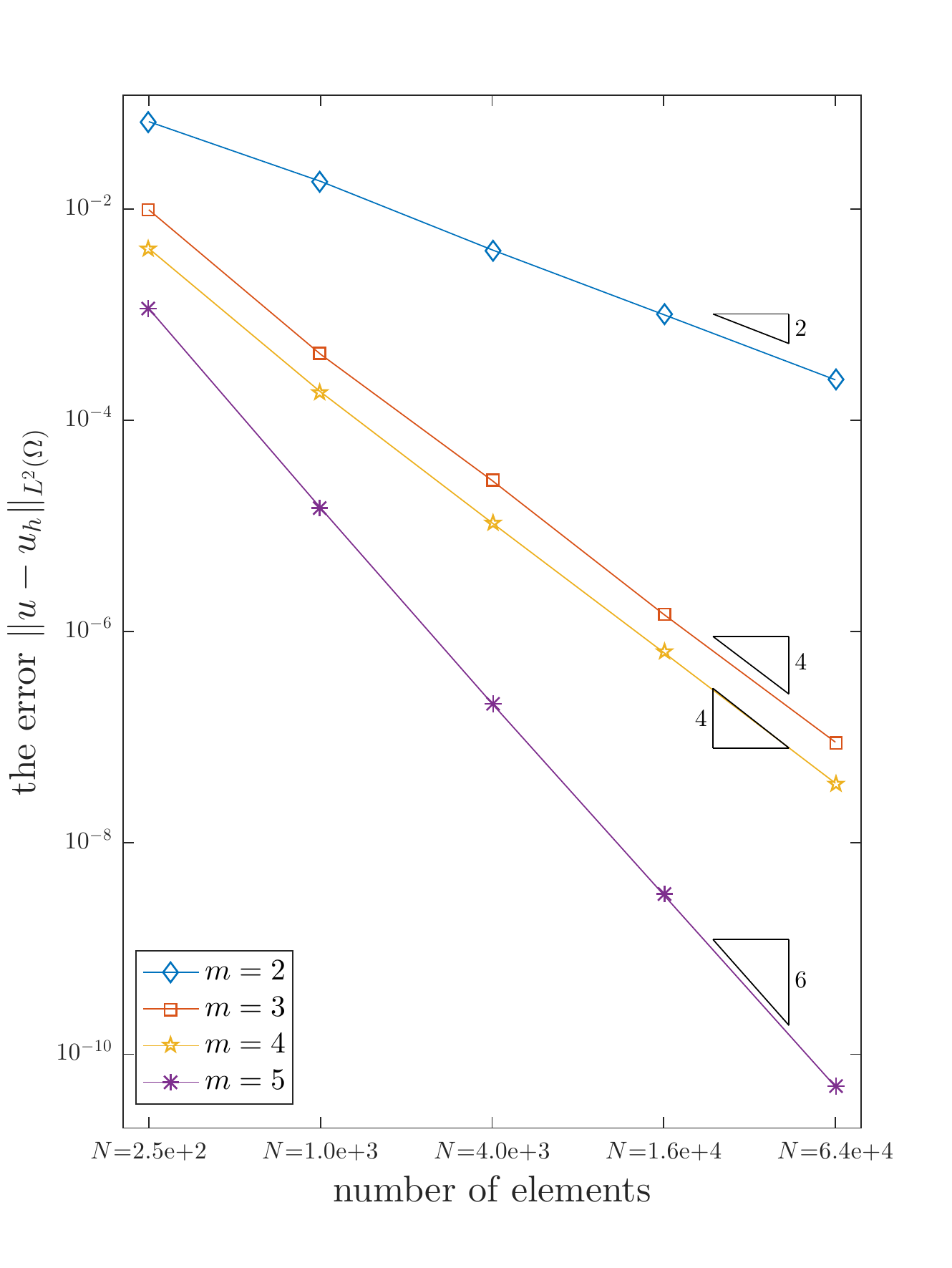}
  \caption{Example 2. The convergence rates of $J_h(\bm{\sigma}_h,
  \bm{u})^{\frac{1}{2}}$ (left) / $\| \bm{\sigma} - \bm{\sigma}_h
  \|_{L^2(\Omega)}$ (mid) / $\| \bm{u} - \bm{u}_h\|_{L^2(\Omega)}$
  (right). }
  \label{fig:ex2}
\end{figure}

\paragraph{\bf Example 3.} In this test, we compute an example in
three dimension. We solve the linear elasticity problem on the unit
cube $\Omega = [0, 1]^3$ and we set $\Gamma_N$ is the boundary with
$y=1$ and $\Gamma_D = \partial \backslash \Gamma_N$
. Let the exact displacement $u$(see \cite{Hu2015family}) be
\begin{displaymath}
  u(x, y, z) = \begin{pmatrix}
    2^4 \\ 2^5 \\ 2^6 \\
  \end{pmatrix} x(1 - x) y (1 - y) z (1 - z).
\end{displaymath}
Then, the load function $\bm{f}$ and the boundary functions $\bm{g},
\bm{h}$ are defined accordingly. The uniform cardinality $\# S(K)$ is
given in Tab.~\ref{tab:numSKd3} for the case $d = 3$. We solve this
example with a series of tetrahedral meshes with mesh size $h = 1/3$,
$1/6$, $1/12$, $1/24$. The Lam\'e parameters $\lambda$, $\mu$ in this
test are $1$. The least squares functional and the errors in various
norms with different $h$ and different $m$ are reported in
Fig.~\ref{fig:ex3}, which distinctly coincide with the theoretical
predicts.

\begin{table}
  \centering
  \renewcommand\arraystretch{1.3}
  \begin{tabular}{p{2.0cm}|p{1.3cm}|p{1.3cm}|p{1.3cm}| p{1.3cm}}
    \hline\hline 
    $m$ & 1 & 2 & 3 & 4  \\
    \hline
    $S(K)$ & 9 & 17 & 35 & 57 \\ 
    \hline\hline
  \end{tabular}
  \caption{$\# S(K)$ for $d=3$.}
  \label{tab:numSKd3}
\end{table}

\begin{figure}
  \centering
  \includegraphics[width=0.32\textwidth]{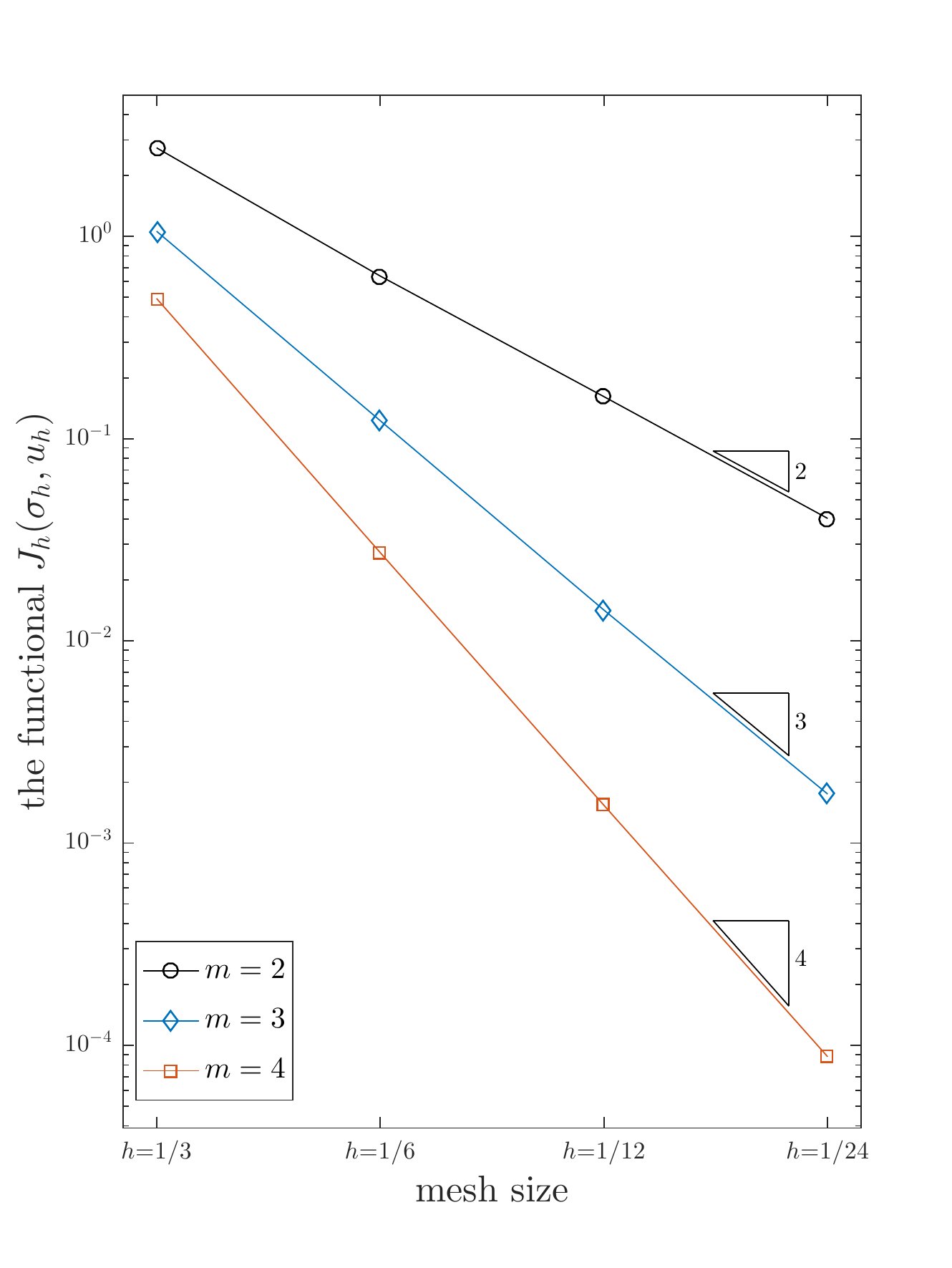}
  \hspace{0pt}
  \includegraphics[width=0.32\textwidth]{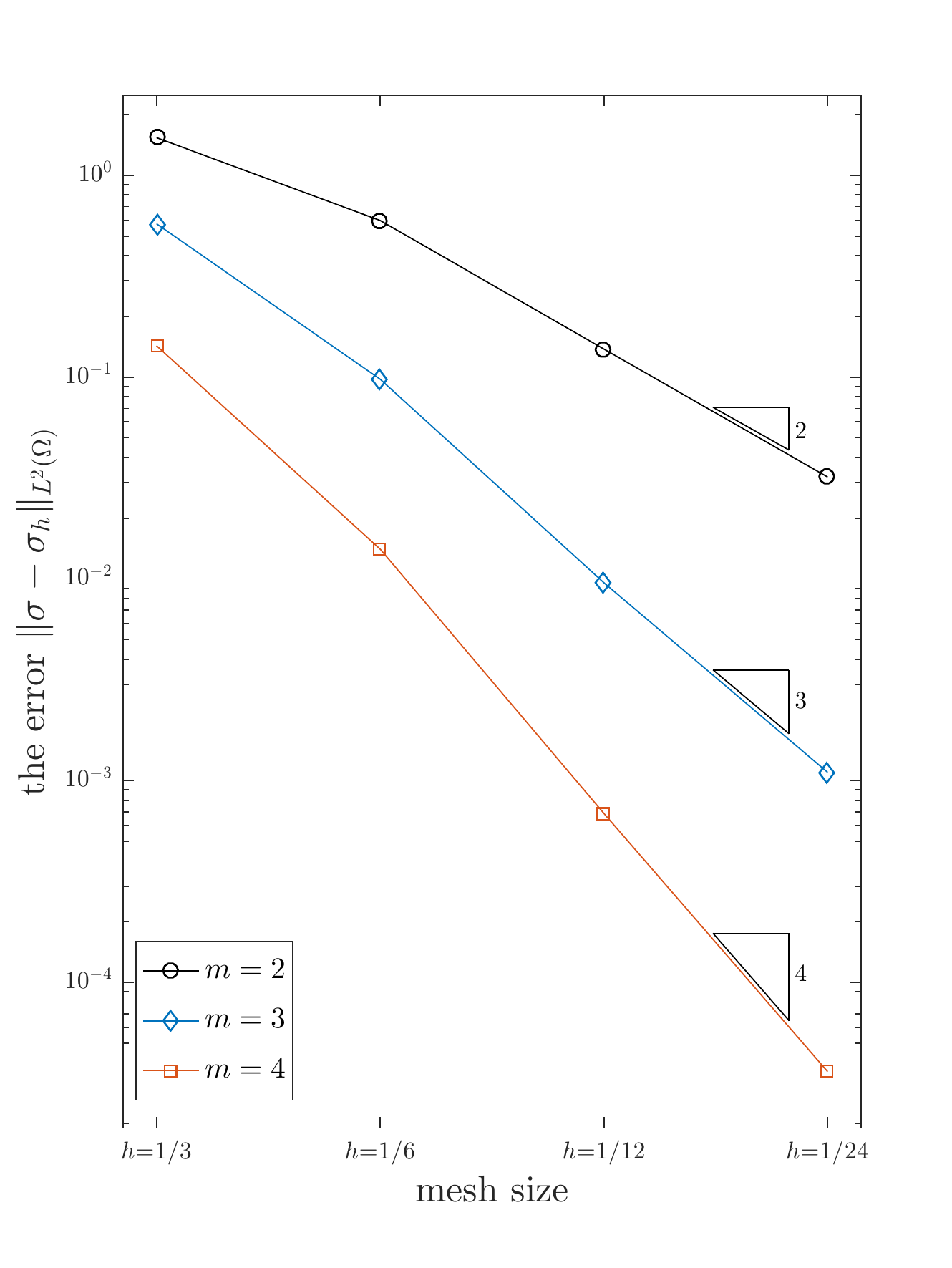}
  \hspace{0pt}
  \includegraphics[width=0.32\textwidth]{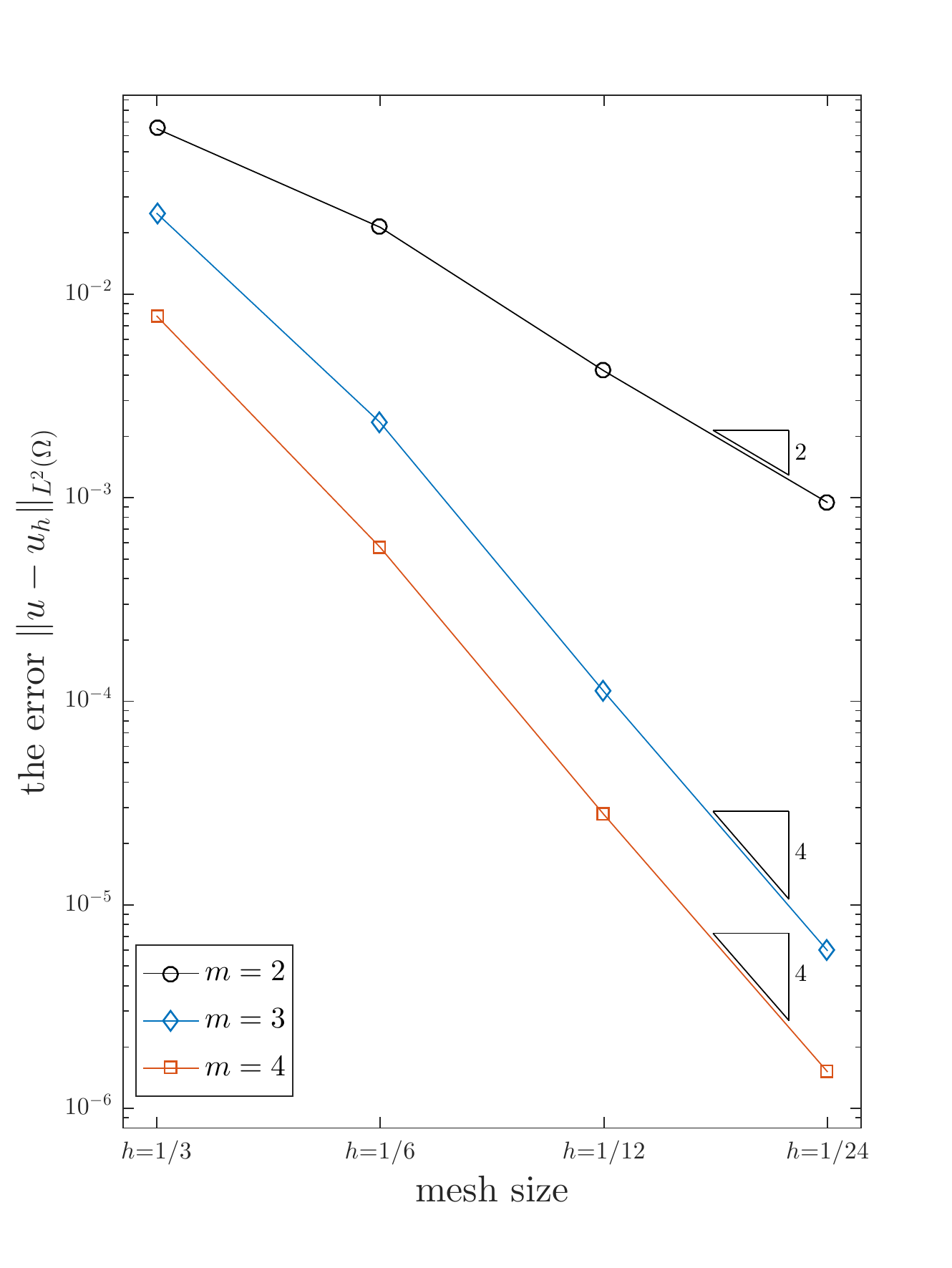}
  \caption{Example 3. The convergence rates of $J_h(\bm{\sigma}_h,
  \bm{u})^{\frac{1}{2}}$ (left) / $\| \bm{\sigma} - \bm{\sigma}_h
  \|_{L^2(\Omega)}$ (mid) / $\| \bm{u} - \bm{u}_h\|_{L^2(\Omega)}$
  (right). }
  \label{fig:ex3}
\end{figure}

\subsection{Efficiency comparison} 
Now let us make a comparison between our method and the classical
continuous least squares method proposed in \cite{Cai2004least}.
According to \cite{hughes2000comparison}, the number of the degrees of
freedom of a specific discrete system could serve as a proper
indicator for the scheme's efficiency. For two dimension, we solve the
problem taken from Example 1 by the two methods on a series of
triangular meshes, respectively, and for three dimension we employ
both methods to solve the problem in Example 3.  Here for the
continuous least squares method, we adopt the standard Lagrange finite
element space of degree $2 \leq m \leq 4$ for each component of the
symmetric stress. To show the efficiency of our method, we compare the
values of $J_h(\cdot, \cdot)$ and $\wt{J}_h(\cdot, \cdot)$, where
$\wt{J}_h(\cdot, \cdot)$ is the least squares functional defined in
\cite{Cai2004least}. The main difference between the two least squares
functional is $\wt{J}_h(\cdot, \cdot)$ contains no jump term defined
on the interior faces.

In Fig.~\ref{fig:compared2} and Fig.~\ref{fig:compared3}, we plot the
values of the least squares functions defined for two methods against
the number of the degrees of freedom with $2 \leq m \leq 4$. All
convergence rates are perfectly consistent with the analysis. Clearly,
our method is more efficient than the continuous least squares method.
To achieve the same accuracy, fewer degrees of freedom are involved in
our method for all $2 \leq m \leq 4$, and the advantage of the
efficiency of our method becomes more prominent with the increasing of
$m$. More specifically, in Tab.~\ref{tab:dofs} we list the ratio
between the number of DOFs in our method and the number of DOFs in
continuous finite element method when the two methods achieve the same
accuracy.  The saving of number of DOFs is more remarkable when
adopting the high-order approximation.

\begin{figure}
  \centering
  \includegraphics[width=0.32\textwidth]{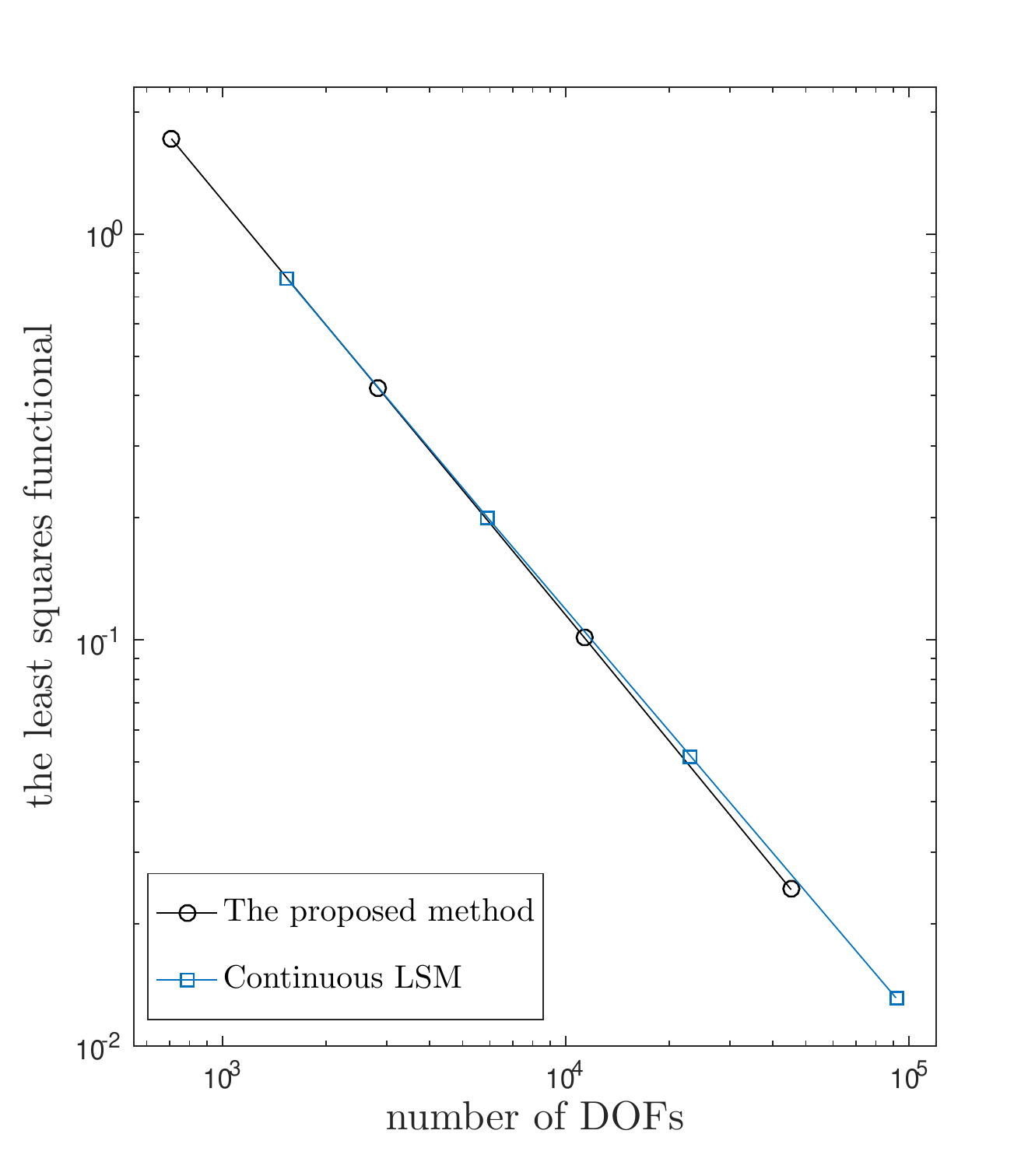}
  \hspace{0pt}
  \includegraphics[width=0.32\textwidth]{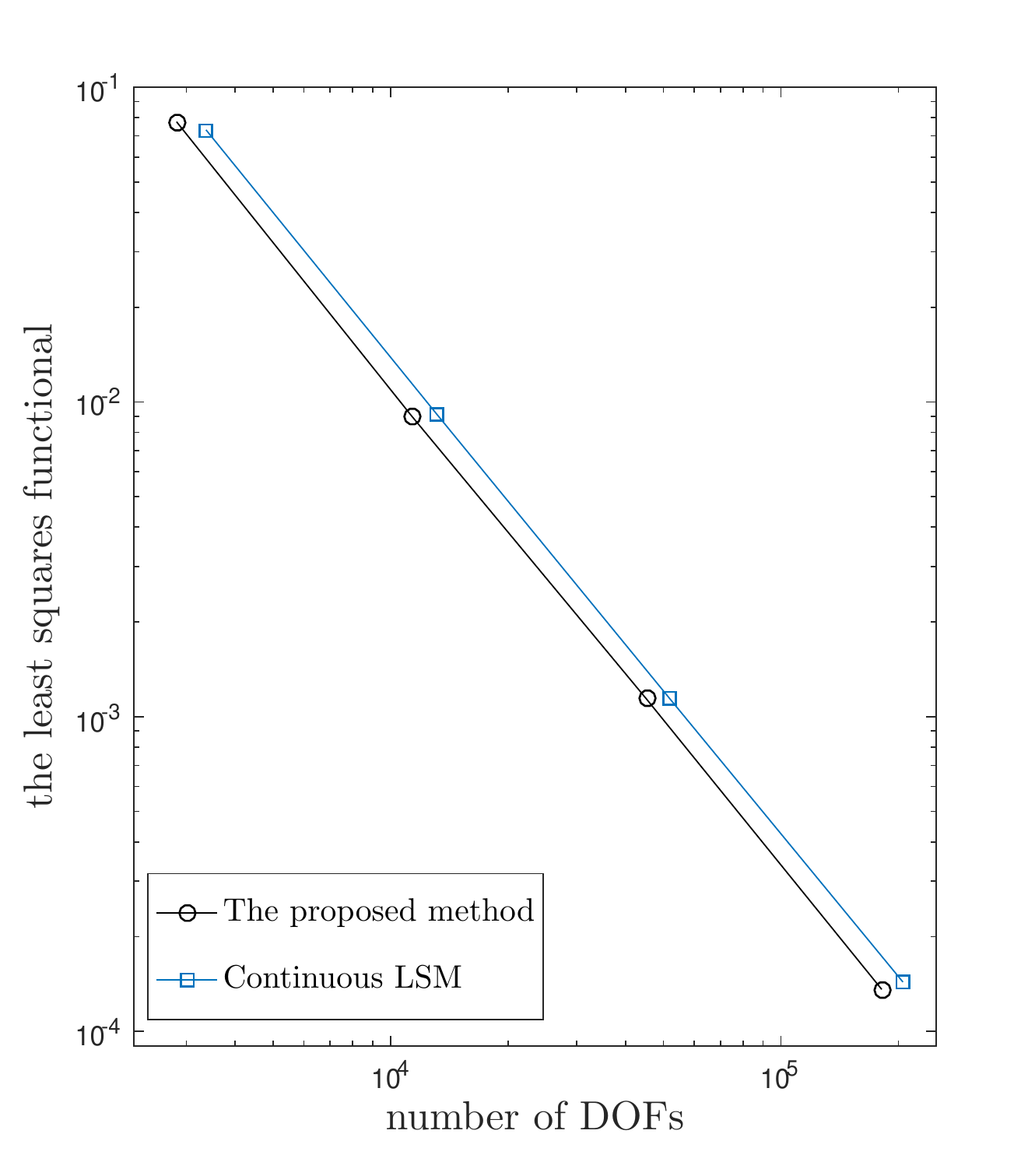}
  \hspace{0pt}
  \includegraphics[width=0.32\textwidth]{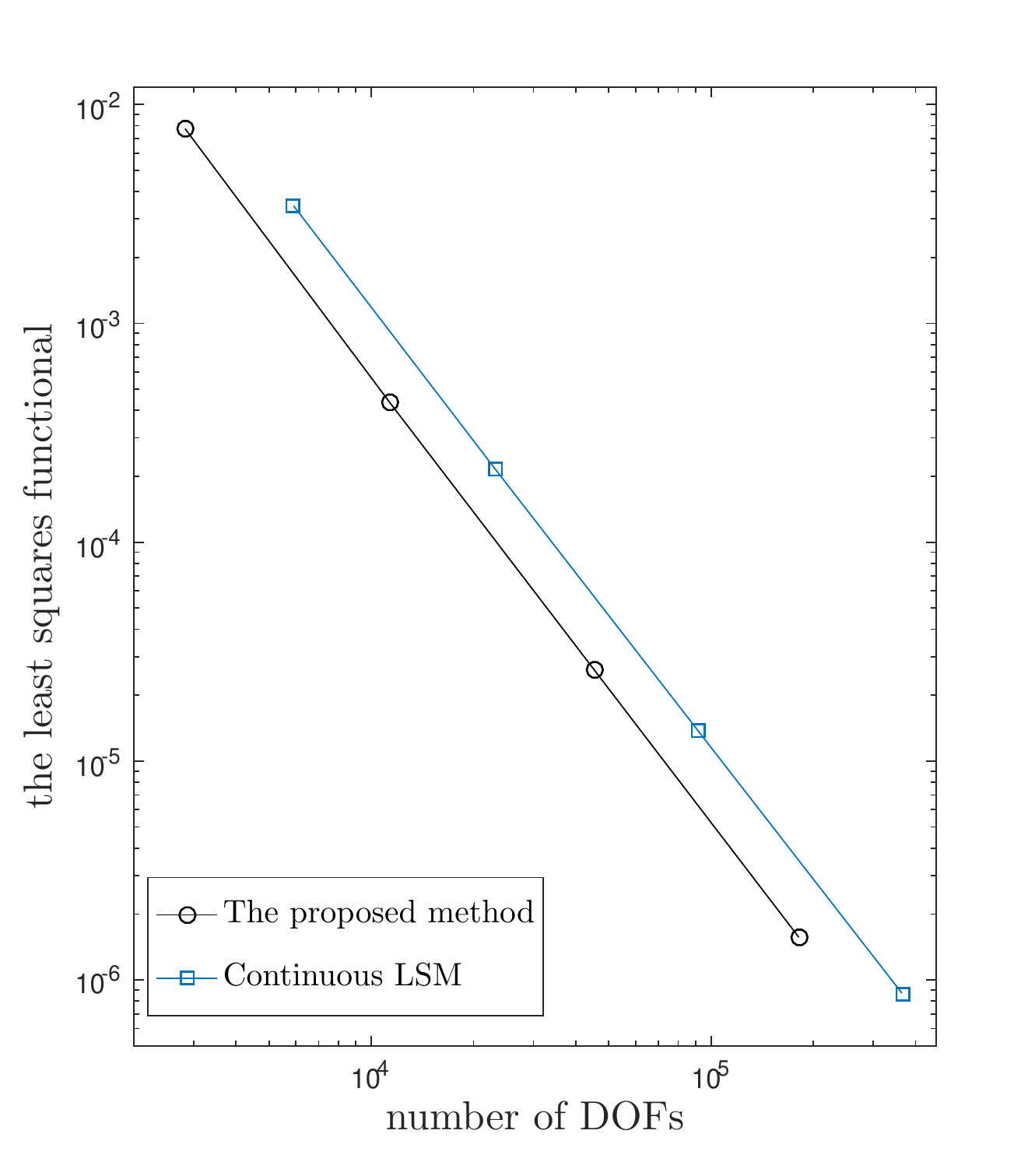}
  \caption{The two-dimensional comparison of efficiency for $m=2$
  (left) / $m=3$ (mid) / $m=4$ (right).}
  \label{fig:compared2}
\end{figure}

\begin{figure}
  \centering
  \includegraphics[width=0.32\textwidth]{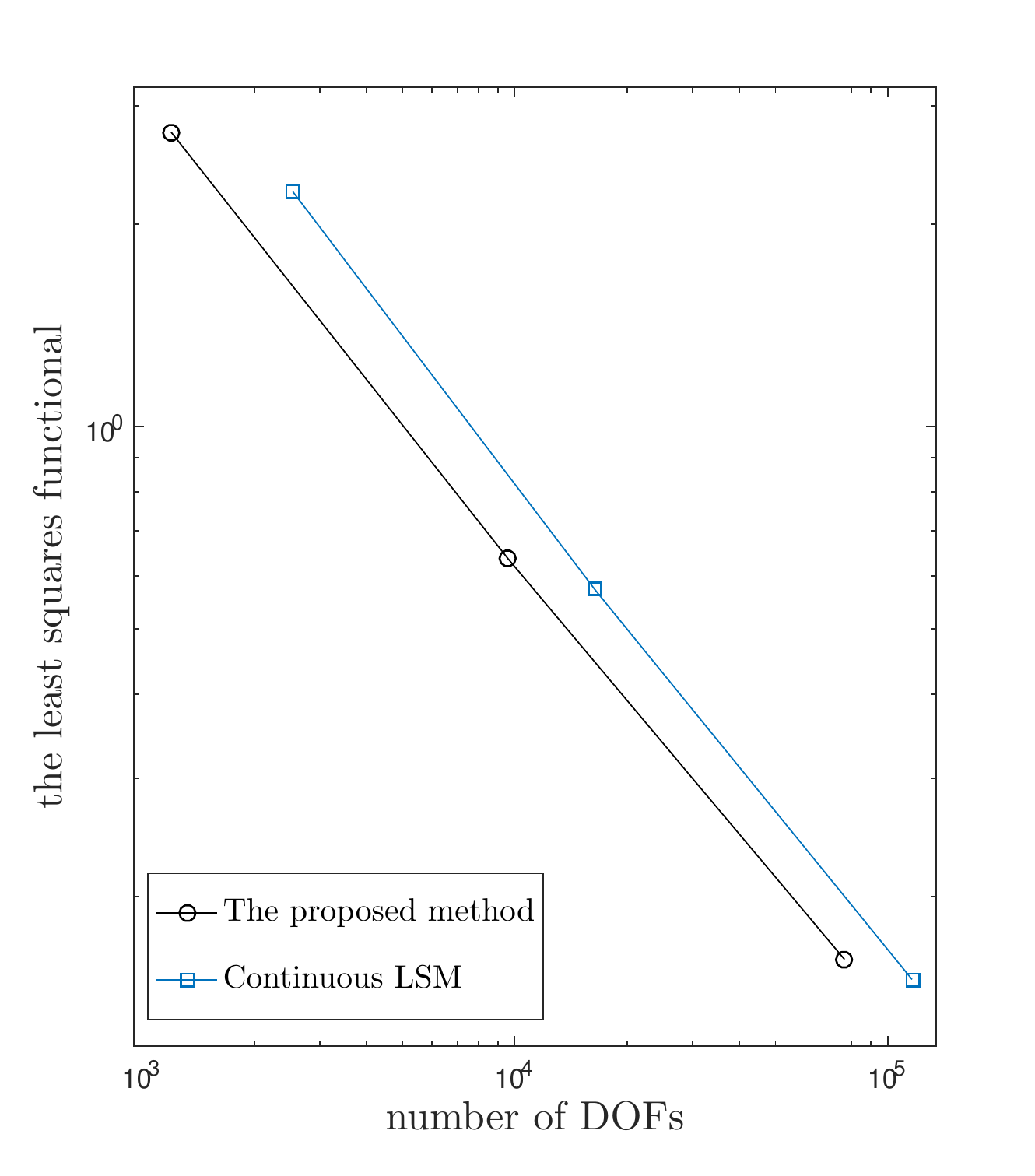}
  \hspace{0pt}
  \includegraphics[width=0.32\textwidth]{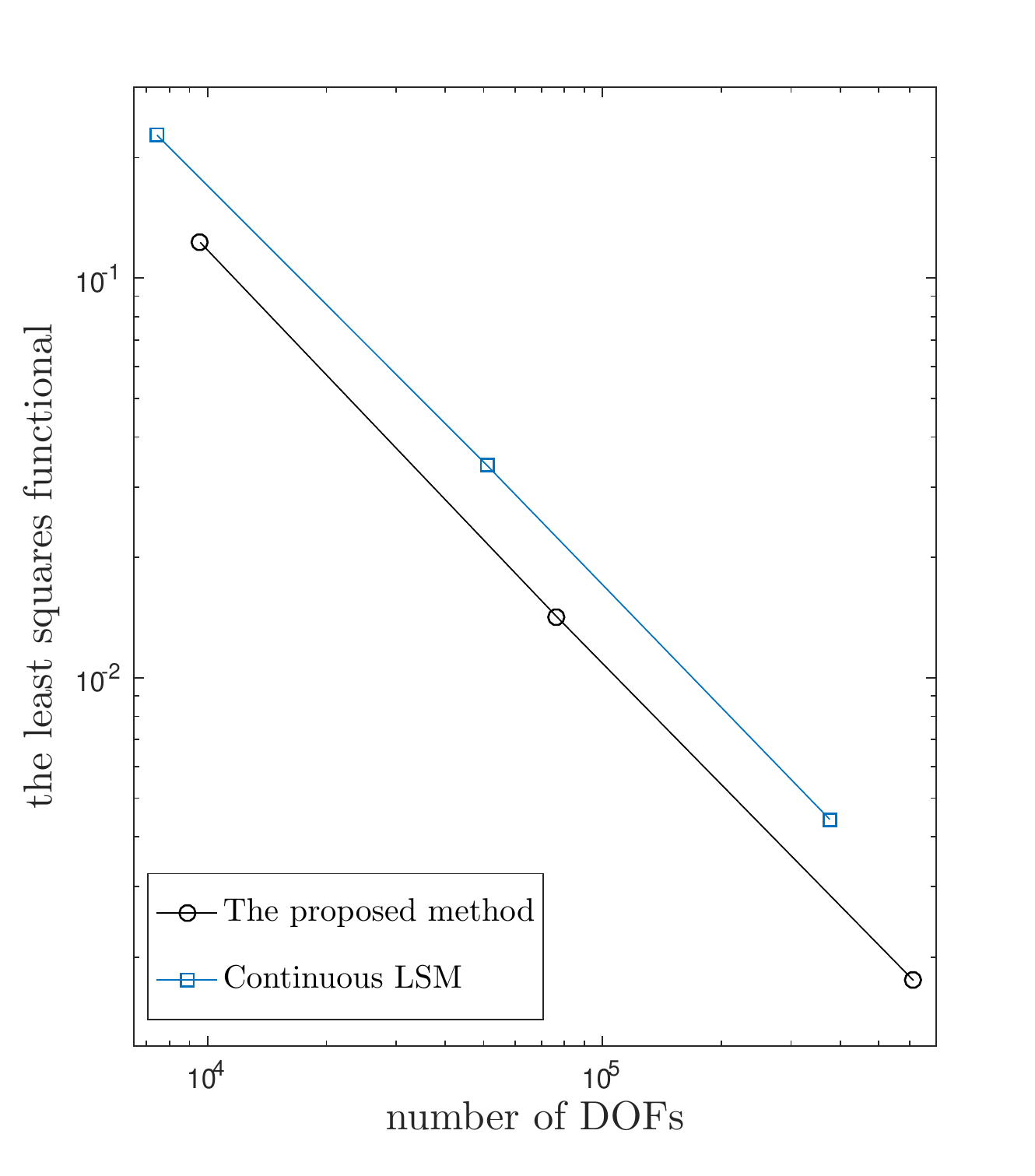}
  \hspace{0pt}
  \includegraphics[width=0.32\textwidth]{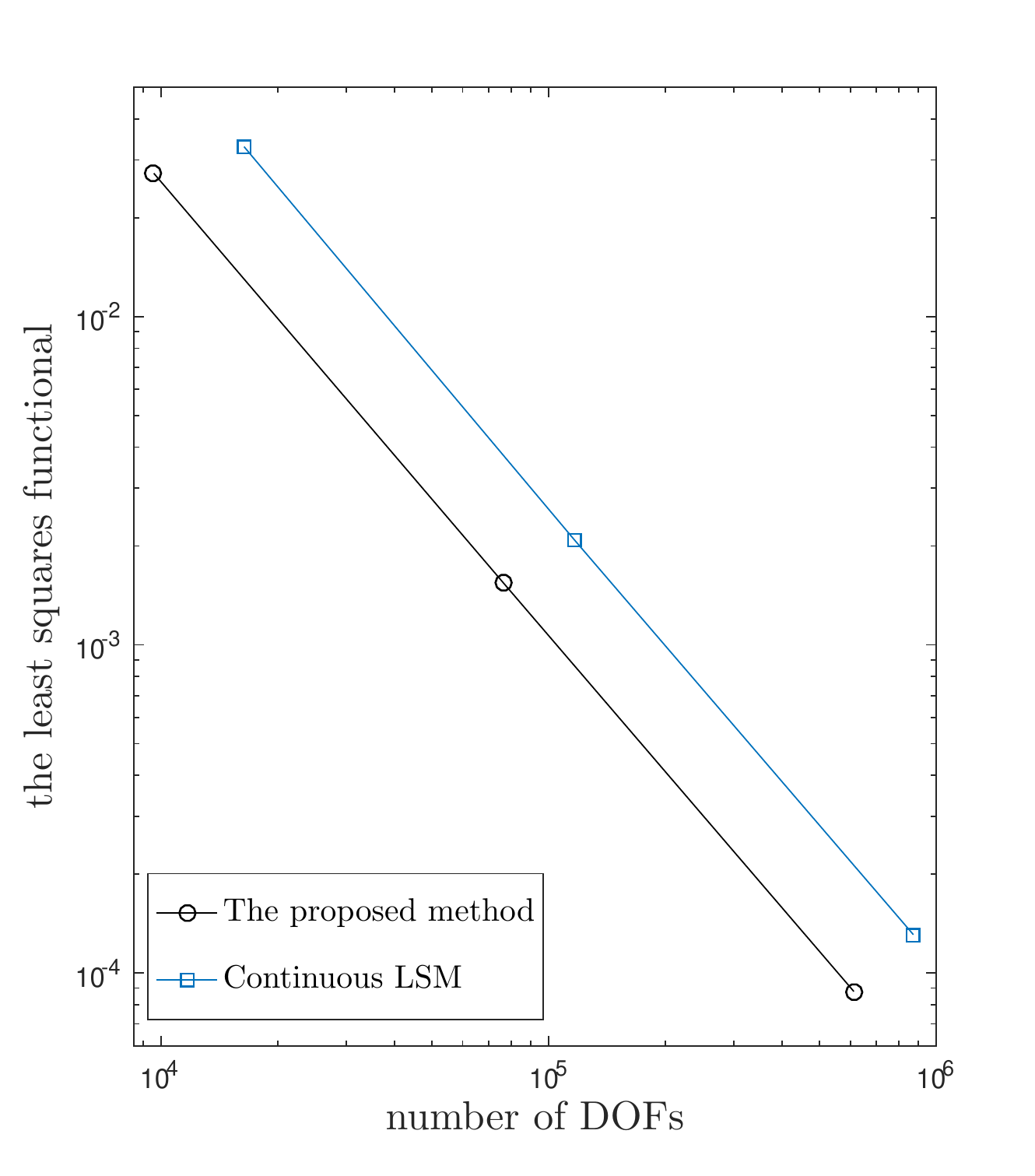}
  \caption{The three-dimensional comparison of efficiency for $m=2$
  (left) / $m=3$ (mid) / $m=4$ (right).}
  \label{fig:compared3}
\end{figure}

\begin{table}
  \centering
  \renewcommand\arraystretch{1.3}
  \begin{tabular}{p{1.5cm}|p{2cm} | p{2cm} | p{2cm} }
    \hline\hline
    & $m=2$ & $m=3$ & $m=4$ \\
    \hline
    $d = 2$ & $96.2\%$  & $86.6\%$  & $68.1\%$ \\
    \hline
    $d = 3$ & $73.8\%$ & $65.6\%$  & $53.8\%$  \\
    \hline\hline
  \end{tabular}
  \caption{To achieve the same accuracy, the ratio of the DOFs
  involved in our method to the DOFs involved in continuous finite
  element method.}
  \label{tab:dofs}
\end{table}

%%% Local Variables:
%%% mode: latex
%%% TeX-master: "linear_elasticity"
%%% End:

% vim:spell:tw=70:fo+=Mn:cc=70
\section{Conclusion}
\label{sec:conclusion}
We proposed a new discontinuous least squares method for the linear
elasticity problem. The approximation space is reconstructed by
solving the local least squares problem. We proved the optimal
convergence order in energy norm. We conducted a sequence of numerical
results that supported our theoretical results and exhibited the great
flexibility, robustness and efficiency of the proposed method.

\section*{Acknowledgements}
This research is supported by the National Natural Science Foundation
of China (Grant No. 91630310, 11421110001, and 11421101) and the
Science Challenge Project, No. TZ2016002.

%%% Local Variables:
%%% mode: latex
%%% TeX-master: "linear_elasticity"
%%% End:

% vim:spell:tw=70:fo+=Mn:cc=70
\begin{appendix}
  \section{Algorithm of constructing element patch}
  \label{sec:acsp}
  In Algorithm~\ref{alg:patch} we show the details of  the algorithm
  of constructing the element patch for every element in partition,
  which is very simple to implement.
  \begin{algorithm}[htb]
    \caption{Constructing Element Patch}
    \label{alg:patch}
    \begin{algorithmic}[1]
    \renewcommand{\algorithmicrequire}{\textbf{Input:}}
    \REQUIRE
    partition $\MTh$ and a uniform threshold $\#
    S(K)$; \\
    \renewcommand{\algorithmicrequire}{\textbf{Output:}}
    \REQUIRE
    the element patch $S(K)$ for all $K$ in partition $\MTh$; \\
    \FOR{each $K \in \MTh$}
    \STATE{set $t=0$, $S_t(K) = \left\{ K \right\}$, $I(K) = \left\{
    \bm{x}_K \right\}$}
    \WHILE{the cardinality of $S_t(K)$ $<$ $\# S(K)$}
    \STATE{initialize sets $S_{t+1}(K) = S_t(K)$}
    \FOR{each $K \in S_t(K)$}
    \STATE{add all adjacent face-neighbouring elements of $K$ to
    $S_{t+1}(K)$}
    \ENDFOR
    \STATE{add the collocation points of all elements in $S_{t+1}(K)$
    to $I(K)$}
    \STATE{let $t = t+1$ and delete $S_t(K)$}
    \ENDWHILE
    \STATE{sort the distances between points in $I(K)$ and $\bm{x}_K$}
    \STATE{select the $\# S(K)$ smallest values and collect the
    corresponding elements to form $S(K)$}
    \ENDFOR
    \end{algorithmic}
  \end{algorithm}
  \section{Solving local least squares problem}
  \label{sec:sllsp}
  In this section, we present an example for solving the local least
  squares problem \eqref{eq:lsproblem}. We choose the linear
  reconstruction as an illustration. For element $K_0$ (the red
  element in Fig.~\ref{fig:Kexample}), we collect $K_0$ and its
  face-neighbouring elements to form the element patch $S(K_0) =
  \left\{ K_0, K_1, K_2, K_3 \right\}$, see Fig.~\ref{fig:Kexample}.
  We let $\mc{I}_{K_0} = \left\{\bm{x}_0, \bm{x}_1, \bm{x}_2, \bm{x}_3
  \right\}$, where $\bm{x}_i$ is the barycenter of $\bm{x}_i$.

  \begin{figure}[htp]
    \centering
    \begin{tikzpicture}[scale=1]
      \coordinate (A) at (1, 0); 
      \coordinate (B) at (-0.5, -0.6);
      \coordinate (C) at (-0.5, 0.8);
      \coordinate (D) at (1.2, 1.5);
      \coordinate (E) at (-2, 0);
      \coordinate (F) at (0.8, -1.5);
      \draw[fill, red] (A) -- (B) -- (C);
      \draw[thick, black] (A) -- (C) -- (B) -- (A);
      \draw[thick, black] (A) -- (D) -- (C);
      \draw[thick, black] (C) -- (E) -- (B);
      \draw[thick, black] (A) -- (F) -- (B);
      \node at(0, 0) {$K_0$}; \node at(1.7/3, 2.3/3) {$K_1$};
      \node at(1.3/3, -2.1/3) {$K_2$}; \node at(-1, 0.2/3) {$K_3$};
    \end{tikzpicture}
    \hspace{70pt}
    \begin{tikzpicture}[scale=1]
      \coordinate (A) at (1, 0); 
      \coordinate (B) at (-0.5, -0.6);
      \coordinate (C) at (-0.5, 0.8);
      \coordinate (D) at (1.2, 1.5);
      \coordinate (E) at (-2, 0);
      \coordinate (F) at (0.8, -1.5);
      \draw[fill, red] (A) -- (B) -- (C);
      \draw[thick, black] (A) -- (C) -- (B) -- (A);
      \draw[thick, black] (A) -- (D) -- (C);
      \draw[thick, black] (C) -- (E) -- (B);
      \draw[thick, black] (A) -- (F) -- (B);
      \node at(0, 0) {$\bm{x}_0$}; \node at(1.7/3, 2.3/3) {$\bm{x}_1$};
      \node at(1.3/3, -2.1/3) {$\bm{x}_2$}; \node at(-1, 0.2/3)
      {$\bm{x}_3$};
    \end{tikzpicture}
    \caption{$K_0$ and its neighbours (left) / barycenters (right).}
    \label{fig:Kexample}
  \end{figure}
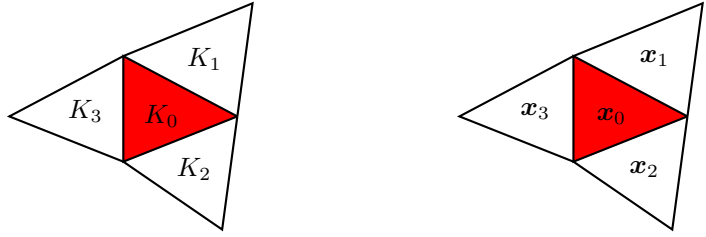
  Then for $g \in C^0(\Omega)$, the least squares problem on $S(K_0)$
  reads:
  \begin{equation}
    \mc{R}_{K_0} g = \mathop{\arg \min}_{ p \in \mb P_1(S({K_0}))}
    \sum_{ \bm{x} \in \mc I_{K_0}} |p(\bm x) - g(\bm x)|^2 \quad
    \text{s.t. } p(\bm{x}_{K_0}) = g(\bm{x}_{K_0}). \\
    \label{eq:K0lsproblem}
  \end{equation}
  Since the constraint of \eqref{eq:K0lsproblem}, $p(\bm{x})$ has the
  form
  \begin{displaymath}
    \begin{aligned}
      p(\bm{x}) &= a + b({x} - {x}_{K_0})  + c(y - y_{K_0}) \\
      p(\bm{x}) &= g(\bm{x}_{K_0}) + b({x} - {x}_{K_0})  + c(y -
      y_{K_0}) \\
    \end{aligned}
  \end{displaymath}
  where $\bm{x} = (x, y)$ and $\bm{x}_i = (x_i, y_i)$. Then, the
  problem \eqref{eq:K0lsproblem} is equivalent to 
  \begin{equation}
    \mathop{\arg \min}_{b, c \in \mb{R}} \sum_{i = 1}^3
    \left|b(x_{K_i} - x_{K_0}) + c(y_{K_i} - y_{K_0}) -\left(
    g(\bm{x}_{K_i}) - g(\bm{x}_{K_0}) \right) \right|^2.
    \label{eq:K0lsproblem2}
  \end{equation}
  It is easy to get the unique solution to \eqref{eq:K0lsproblem2}:
  \begin{displaymath}
    \begin{bmatrix}
      b \\ c \\
    \end{bmatrix} = M\begin{bmatrix}
      g(\bm{x}_{K_1}) - g(\bm{x}_{K_0}) \\
      g(\bm{x}_{K_2}) - g(\bm{x}_{K_0}) \\
      g(\bm{x}_{K_3}) - g(\bm{x}_{K_0}) \\
    \end{bmatrix},
  \end{displaymath}
  where 
  \begin{displaymath}
    M = (A^TA)^{-1}A^T, \quad
    A = \begin{bmatrix}
      x_{K_1} - x_{K_0} & y_{K_1} - y_{K_0} \\
      x_{K_2} - x_{K_0} & y_{K_2} - y_{K_0} \\
      x_{K_3} - x_{K_0} & y_{K_3} - y_{K_0} \\
    \end{bmatrix}.
  \end{displaymath}
  Hence,
  \begin{displaymath}
    \begin{bmatrix}
      a \\ b \\ c \\
    \end{bmatrix} = \begin{bmatrix}
      1 & 0 \\
      -MI_{3 \times 1} & M \\
    \end{bmatrix}\begin{bmatrix}
      g(\bm{x}_{K_0}) \\ 
      g(\bm{x}_{K_1}) \\ 
      g(\bm{x}_{K_2}) \\ 
      g(\bm{x}_{K_3}) \\ 
    \end{bmatrix},
  \end{displaymath}
  where $I_{3\times 1} = (1,1,1)^T$. It is noticeable that the matrix
  $M$ is independent on the function $g$ and contains all information
  of the function $\lambda_{K_0}$, $\lambda_{K_1}$, $\lambda_{K_2}$,
  $\lambda_{K_3}$ on the element $K_0$ according to the expression
  \eqref{eq:explicitly}. Then we store the matrix $M$ on all elements
  to represent the approximation space $U_h$. The procedure of this
  implementation could be adapted to the case of greater $m$ without
  difficulties.
\end{appendix}

%%% Local Variables:
%%% mode: latex
%%% TeX-master: "linear_elasticity"
%%% End:

\bibliographystyle{amsplain}
\bibliography{../ref}

\end{document}